\documentclass[a4paper, reqno]{amsart}

\usepackage[utf8]{inputenc}
\usepackage[T1]{fontenc}
\usepackage{amsmath, amssymb, amscd, amsthm,epsfig,epic,eepic}
\usepackage[english]{babel}
\usepackage{mathrsfs}
\usepackage{tikz-cd}
\usepackage{verbatim}
\usepackage{color}
\usepackage[all]{xy}
\usepackage[justify]{ragged2e}

\usepackage[shortlabels]{enumitem}
\setlist[enumerate]{label=\rm{(\arabic*)}}
\setlist[enumerate,2]{label=\rm{(\roman*)}}
\setlist[itemize]{label=\raisebox{0.25ex}{\tiny$\bullet$}}
\renewcommand{\thefootnote}{\fnsymbol{footnote}}

\usepackage[left=3cm,right=3cm,top=3cm,bottom=3cm]{geometry}

\makeatletter
     \let\oldfootnote\footnote
     \def\footnote{\@ifstar\footnote@star\footnote@nostar}
     \def\footnote@star#1{{\let\thefootnote\relax\footnotetext{#1}}}
     \def\footnote@nostar{\oldfootnote}
     \makeatother

\newcommand{\N}{\mathbf{N}}

\newcommand{\red}{\mathrm{red}}
 
\makeatletter
\newcommand*{\defeq}{\mathrel{\rlap{%
                      \raisebox{0.3ex}{$\cdot$}}%
                      \raisebox{-0.3ex}{$\cdot$}}%
                      =}
\makeatother

\newcommand{\proj}{\mathbf{P}}
\newcommand{\Gm}{\mathbf{G}_m}
\newcommand{\Ga}{\mathbf{G}_a}

\DeclareMathOperator{\Proj}{Proj}

\DeclareMathOperator{\Aut}{Aut}

\DeclareMathOperator{\Supp}{Supp}

\DeclareMathOperator{\height}{ht}
\DeclareMathOperator{\Lie}{Lie}

\DeclareMathOperator{\PSO}{PSO}

\DeclareMathOperator{\PGL}{PGL}
\DeclareMathOperator{\SO}{SO}
\DeclareMathOperator{\Sp}{Sp}
\DeclareMathOperator{\PSp}{PSp}

\DeclareMathOperator{\Pic}{Pic}

\DeclareMathOperator{\diag}{diag}

\DeclareMathOperator{\sm}{sm}

\usepackage{xcolor, color}
\usepackage{hyperref, cleveref}
\hypersetup{
    colorlinks=true,
    linkcolor=teal,
    citecolor=violet,
    filecolor=violet,
    urlcolor=violet,
    linktocpage=true
}
 
\theoremstyle{plain}

\newtheorem{theorem}{Theorem}[section]
\newtheorem{lemma}[theorem]{Lemma}
\newtheorem{proposition}[theorem]{Proposition}
\newtheorem{corollary}[theorem]{Corollary}
\newtheorem{theorem*}{Theorem}
\newtheorem{proposition*}[theorem*]{Proposition}
\newtheorem{lemma*}[theorem*]{Lemma}
\newtheorem*{proposition**}{Proposition}
\newtheorem*{conventions*}{Conventions}
\newtheorem{comment*}{Comment}
 
\newtheorem*{Acknowledgements}{Acknowledgements} 
 
\theoremstyle{definition}
\newtheorem{definition}[theorem]{Definition}
\newtheorem{remark}[theorem]{Remark}

\newtheorem{example}[theorem]{Example}
\numberwithin{equation}{section}
\newtheorem{theoremA}{Theorem}

\begin{document}

\title[Automorphism group of flag varieties]{Automorphism group of flag varieties\\ with non-reduced stabilizer}

\author{Matilde Maccan}
    \address{Fakultät für Mathematik, RUB, Universitätsstr. 150, 44801 Bochum, Germany}%
	\email{matilde.maccan@ruhr-uni-bochum.de}

\footnote*{Keywords: flag variety, parabolic subgroup, automorphism group, group schemes}
\footnote*{2020 Mathematics Subject Classification: 14M15, 14L15, 14L30, 17B20}
\date{\today}

\begin{abstract}
We consider rational projective homogeneous varieties over an algebraically closed field of positive characteristic, namely quotients of a semi-simple group by a possibly non-reduced parabolic subgroup. We determine the group scheme structure of the neutral component of their automorphism group, generalizing the classical result of Demazure on automorphism groups of flag varieties.
\end{abstract}

\maketitle

\addtocontents{toc}{\setcounter{tocdepth}{1}}
\tableofcontents

\section{Introduction}

Our main objects of interest are projective rational homogeneous varieties, namely quotients of a semisimple algebraic group $G$ by a \emph{parabolic subgroup}. When the stabilizer of a point is smooth, these are known \emph{flag varieties}, which form a well-studied class of objects in algebraic geometry. The root system of $G$ provides a useful combinatorial framework, allowing for explicit computations involving representation-theoretical tools. 

\medskip

Let us fix a Borel subgroup \( B \) with unipotent radical \( U \), along with a maximal torus \( T \subset B \). Up to conjugation, a parabolic subgroup is defined as a subgroup scheme of \( G \) containing \( B \). All objects and morphisms, unless otherwise specified, are defined over an algebraically closed base field \( k \), of characteristic \( p > 0 \).  Throughout, we work with group schemes of finite type over \( k \), so all subgroups are understood as subgroup \emph{schemes}, not necessarily smooth.

\medskip It is well known that, over an algebraically closed field, reduced parabolic subgroups containing \( B \) correspond bijectively to subsets of the set of simple roots of \( G \): a parabolic subgroup is uniquely determined by the basis of the root system of a Levi subgroup. However, over fields of positive characteristic, parabolic subgroups \emph{can be non-reduced} (equivalently, non-smooth). This leads to a richer structure and geometric behavior for the corresponding homogeneous spaces; see, for example, \cite[Section 4]{HL93}, \cite{Lauritzen}, and \cite[Theorem 3.1]{Totaro}. The classification of non-reduced parabolics was established in \cite{Wenzel} for characteristic at least 5, and was recently completed in small characteristics in \cite{Maccan, Maccan2}.

\medskip

A key ingredient in the theory of algebraic groups in positive characteristic, and even more in the classification of non-reduced parabolics, is the kernel of the \( m \)-th iterated (relative) Frobenius homomorphism of \( G \), which we denote as \( {}_mG \). In particular, Frobenius kernels are infinitesimal subgroups: by construction, their reduced part is trivial, making them, in a sense, the opposite notion to reduced group schemes. For background on group schemes, we refer to \cite{DG, Milne}.

\medskip

The central question addressed in this work is: what is the connected component of the automorphism group of \( G/P \), \emph{as a group scheme}? Its reduced subgroup is semisimple of adjoint type, but it might contain $G$ strictly. Since any projective rational homogeneous variety writes as a product of such varieties under \emph{simple adjoint} groups, we can and will assume that $G$ is simple adjoint. By the work of Demazure \cite{Demazure}, it is known that if \( P \) is reduced, then this automorphism group coincides with \( G \), except for three exceptional cases, all corresponding to varieties of Picard rank one. However, some examples in positive characteristic show that this automorphism group can be non-reduced. To the author's knowledge, the only example treated in the literature is \cite[Proposition 4.3.4]{BSU}, where non-reducedness is established via the dimension of the Lie algebra, though the precise structure of the group scheme remains undetermined.

\medskip

For a subset \( I \) of the simple roots of \( G \), we denote by \( P_I \) the corresponding reduced parabolic subgroup, whose Levi subgroup is generated by \( I \). For a simple root \( \alpha \), let \( P^\alpha \) be the maximal reduced parabolic subgroup that has trivial intersection with \( U_{-\alpha} \).

\medskip
Throughout the present text, we refer to the following pairs \((G, \alpha)\), where \( G \) is a simple adjoint group and \( \alpha \) a simple root, as \emph{exceptional pairs}, in the sense of Demazure's original paper \cite{Demazure}:
\begin{itemize}
    \item \( G = \PSp_{2n} \) with its first simple root \( \alpha_1 \);
    \item \( G = \SO_{2n+1} \) with the unique short simple root \( \alpha_n \);
    \item \( G \) of type \( G_2 \) with its short simple root \( \alpha_1 \).
\end{itemize}

The corresponding homogeneous spaces \( G/P^\alpha \) are well-known geometric objects: respectively, the projective space \( \proj^{2n-1} \); the Grassmannian of totally isotropic \( n \)-dimensional subspaces in an orthogonal vector space of dimension \( 2n+1 \); and a smooth quadric in \( \proj^6 \).

\medskip

We begin with the case of Picard rank one. By \cite{Maccan}, up to isomorphism, there is exactly one homogeneous variety not of the form \( G/P^\beta \) for some simple root \( \beta \). This variety exists only in characteristic \( p = 2 \) and is homogeneous under a group \( G \) of type \( G_2 \). It can be realized as a general hyperplane section, with respect to the unique ample generator of the six-dimensional Lagrangian Grassmannian. In \Cref{rank one exotic}, we show that the automorphism group of this variety is precisely \( G \), so no new phenomena occur in this case. For the second exotic parabolic subgroup, again for a group of type \( G_2 \) in characteristic two, its associated homogeneous space is \( \proj^5 \), with automorphism group \( \PGL_6 \supsetneq G \).

\medskip
The case of higher Picard ranks remains to be addressed. In this context, Schubert curves are of fundamental importance. These curves are smooth, rational, and \( B \)-stable, and they generate the cone of $1$-cycles on \( X \). Moreover, they are dual to the Schubert divisors with respect to the intersection pairing, and in particular the number of such curves equals the Picard rank of \( X \). As recalled in \Cref{sec contractions} below, the stabilizer of a point in \( X \) can be described via contractions of Schubert curves. By a contraction, we mean a morphism of varieties whose direct image of the structure sheaf coincides with the structure sheaf of the target. Such morphisms are significant in our context because, by Blanchard's Lemma \cite[$7.2$]{Brion17}, any action of a connected algebraic group on the source of a contraction $f$ induces a unique action on the target such that $f$ is equivariant.

\medskip

Combining these contractions with Demazure's three exceptional pairs, the group scheme structure of the automorphism group of \( X = G/P \) can be determined, as stated in \Cref{thm main geometric} below, which is the main result of the present work.

\medskip  The key point is that any non-reduced parabolic subgroup \( P \) (up to taking the quotient by the kernel of some purely inseparable isogeny, which does not affect the underlying homogeneous variety) is contained in a unique maximal reduced parabolic \( P^{\text{sm}} \), except in certain exotic cases for groups of type \( G_2 \) when \( p=2 \). The corresponding quotient map 
\[
f \colon X \longrightarrow  G/P^{\text{sm}} \]
can be described as the smooth contraction of Schubert curves on \( X \) that is of maximal relative Picard rank and does not contract \( X \) to a point; more details can be found in \Cref{sec contractions}. 


\begin{theoremA}
\label{thm main geometric} Let $X$ be a homogeneous variety under a simple group $G$ of adjoint type, of Picard rank at least two. Then the inclusion $G \hookrightarrow \underline{\Aut}_X^0$ is an isomorphism, unless: 
        \begin{enumerate}[(a)]
            \item either the smooth contraction $f$ exists, and its target is of the form $ Y = G/P^\alpha$ where $(G,\alpha)$ is an exceptional pair;
            \item or $p=2$, the group $G$ is of type $G_2$, and one of the two contractions of Schubert curves with source $X$ has target $Y = \proj^5$.
        \end{enumerate}
    Under these assumptions, there is some $m \geq 1$ such that
    \[
    G = (\underline{\Aut}_X^0)_{\text{red}} \subsetneq {}_m (\underline{\Aut}_Y^0) \cdot G \subsetneq \underline{\Aut}_Y^0.
    \]
\end{theoremA}

Let us reformulate \Cref{thm main geometric} in a way that is less geometric but more convenient for the proofs. We split the statement into two parts, corresponding to the conditions $(a)$ and $(b)$ in the statement. These will be proven independently, to account for the existence of exotic parabolic subgroups in characteristic two for groups of type \( G_2 \).

\begin{theoremA}
\label{aut main intro}
    Let $X = G/P$ where $G$ is a simple group of adjoint type and $P$ is a parabolic subgroup of the form
    \[
    P = P_J \cap (\ker \xi)P^\prime,
    \]
    where the roots not in the Levi of $P_J$ and those not in the Levi of $P^\prime_{\text{red}}$ form distinct subsets of the basis.
    Moreover, assume that the noncentral isogeny $\xi$ is minimal with respect to inclusion among those satisfying such an equality. Then the following hold.
    \begin{enumerate}
        \item If $P_J$ is not of the form $P^\alpha$ for an exceptional root $\alpha$, then $\underline{\Aut}_X^0 = G$.
        \item If $P_J = P^\alpha$ for some exceptional root $\alpha$, let $m \geq 0$ such that $ker \xi$ contains the Frobenius kernel ${}_mG$ but not ${}_{m+1}G$. Then $\underline{\Aut}_X^0 = {}_m\hat{G} \cdot G$, where $\hat{G} = \underline{\Aut}^0_{G/P^\alpha}$.
    \end{enumerate}
\end{theoremA}

\begin{theoremA}
\label{G2 intro}
    Let $p=2$ and $G$ be a group of type $G_2$. Let $X=G/P$ where $P$ is a parabolic subgroup not containing the Frobenius kernel of $G$. Assume that $X$ is not isomorphic to any of the varieties considered in \Cref{aut main intro}. Then the inclusion $G \hookrightarrow \underline{\Aut}_X^0$ is an isomorphism, unless 
    \[
    P = \mathcal{Q}_1 \cap {}_m GP^{\alpha_2},
    \]
    where $\mathcal{Q}_1$ is an exotic parabolic subgroup such that $G/\mathcal{Q}_1 = \proj^5$ and $\alpha_2$ is the long simple root; in this case,
    \[
    G = (\underline{\Aut}_X^0)_{\text{red}} \subsetneq {}_m (\PGL_6) \cdot G \subsetneq \PGL_6.
    \]
\end{theoremA}

\smallskip

The present text is structured as follows. In \Cref{sec parabolic groups}, the complete classification of parabolic subgroup schemes of a simple adjoint group \( G \) is briefly illustrated. In \Cref{sec Demazure}, we recall the description of the automorphism group in the case of a reduced parabolic and make some observations regarding Demazure's exceptional pairs. These results are then applied in \Cref{sec automorphism} to prove \Cref{aut main intro}. Finally, \Cref{sec exotic} addresses the exotic parabolic subgroups in characteristic two, establishing \Cref{G2 intro} and thereby completing the proof of \Cref{thm main geometric}.

\begin{Acknowledgements}
   I wish to express my gratitude to Michel Brion, Pierre-Emmanuel Chaput, Andrea Fanelli, Bianca Gouthier, Gebhard Martin, Matthieu Romagny, Ronan Terpereau and Dajano Tossici, for all the encouragement and the various stimulating mathematical discussions. This work was supported by the DFG through the research grants Le 3093/5-1 and Le 3093/7-1 (project numbers 530132094; 550535392).
\end{Acknowledgements}

\section{Parabolic subgroup schemes}\label{sec parabolic groups}
Let us start by reviewing the complete classification of parabolic subgroups, building on classical results in the theory of reductive groups, along with the works of \cite{Wenzel,HL93} under the assumption \( p \geq 5 \), and \cite{Maccan,Maccan2} covering the remaining cases in small characteristic.

\medskip

Let \( k \) be an algebraically closed field of characteristic \( p > 0 \). Let \( G \) be a semisimple algebraic group of adjoint type, and fix a Borel subgroup \( B \) with unipotent radical \( U \), along with a maximal torus \( T \subset B \).

\medskip  

We denote by \( \Phi \supset \Phi^+ \supset \Delta \) the sets of roots, positive roots, and simple roots of \( G \) associated with \( B \). For each root \( \gamma \), let \( U_\gamma \) be the corresponding root subgroup, \( u_\gamma \colon \Ga \to U_\gamma \) the associated root homomorphism, and \( \mathfrak{g}_\gamma \) the corresponding root subspace in \( \operatorname{Lie} G \). For a simple root \( \alpha \), denote by \( P^\alpha \) the maximal reduced parabolic subgroup that does not contain \( U_{-\alpha} \). The \( m \)-th iterated (relative) Frobenius homomorphism is written as
\[
F^m = F^m_G \colon G \longrightarrow G^{(m)}
\]
with kernel \( {}_mG \). Throughout, we work with group schemes of finite type over \( k \), so all subgroups are understood as subgroup \emph{schemes}, not necessarily smooth.

\subsection{The associated function} 
Let \( \widetilde{G} \) denote the simply connected cover of \( G \). In this setting, any parabolic subgroup of \( \widetilde{G} \) decomposes as a product of parabolics of the corresponding simple factors; \cite[Lemma 1.2]{Maccan2} for details. Moreover, the connected automorphism group of \( G/P \times G'/P' \) is the product of the automorphism groups of each factor. Finally, the center of \( \widetilde{G} \), where \( \sigma \colon \widetilde{G} \to G \) is the covering map, is contained in every parabolic subgroup. This allows us to assume, without loss of generality, that \( G \) is \emph{simple and adjoint}.

\medskip

By the \emph{height} \( \height(-) \) of a finite subgroup, we mean the smallest integer \( m \) such that the \( m \)-th iterated Frobenius on it is trivial; if no such \( m \) exists, the height is defined to be infinite. Note that for a subgroup \( K \subset G \) and a positive root \( \delta \), the condition that \( K \cap U_\delta \) has infinite height is equivalent to \( U_\delta \) being contained in \( K \). The following is a fundamental structure result describing groups of height at most \( 1 \) (that is, groups killed by the relative Frobenius) in terms of their Lie algebras: \cite[II, §7, n°4]{DG}.

\begin{theorem}
	\label{DG height one}
	Let $G$ be an algebraic group over $k$. Then
	\begin{enumerate}
		\item The map $M \to \Lie M$ induces a bijection from the set of \emph{subgroups of $G$ of height $\leq 1$} and the set of \emph{$p$-Lie subalgebras of $\Lie G$}. Under this bijection, normal subgroups correspond to $p$-Lie ideals.
		\item If $M$ and $M^\prime$ are two subgroups of $G$ and if $M$ is killed by the Frobenius morphism, then $M \subset M^\prime$ is equivalent to $\Lie M \subset \Lie M^\prime$.
		\item Any two homomorphism $f_1$ and $f_2$ with source ${}_1G$ satisfy $f_1 = f_2$ if and only if $\Lie f_1 = \Lie f_2$.
        \item Let $V$ be a finite-dimensional representation of $G$; then a vector subspace of $V$ is ${}_1G$-invariant if and only if it is $\Lie G$-invariant.
		\item If $M$ is a subgroup of $G$ killed by the Frobenius morphism, then $\Lie N_G(M) = N_{\Lie G} (\Lie M).$
	\end{enumerate}
\end{theorem}

\noindent For a non-reduced parabolic subgroup $P$ with reduced part $P_{\text{red}}$, define
\begin{align}
\label{infinitesimal}
U_P^- \defeq P \cap R_u(P_{\text{red}}^-)
\end{align}
as its intersection with the unipotent radical of the opposite reduced parabolic $P_{\text{red}}^-$. The subgroup $U_P^-$ is unipotent, infinitesimal, and satisfies
\begin{align}
\label{product}
U_P^- = \prod_{\gamma \in \Phi^+ \backslash \Phi_I} (U_P^- \cap U_{-\gamma}) \quad \text{and} \quad P = U_P^- \times P_{\text{red}},
\end{align}
where both identities are scheme isomorphisms given by the multiplication in $G$. Thus, $P$ is determined by its reduced part $P_{\text{red}}$, together with its intersections with the root subgroups contained in the opposite unipotent radical $R_u(P_{\text{red}}^-)$. This can be reformulated in a more combinatorial way by introducing a numerical function. The kernel of the $n$-th iterated Frobenius of the additive group $\Ga$ is denoted by $\boldsymbol{\alpha}_{p^n}$; by convention, $\boldsymbol{\alpha}_{p^\infty} = \Ga$.

\begin{definition}
\label{def_varphi}
    Let $P$ be a parabolic subgroup of $G$. The \emph{associated function}
    \[
    \varphi \colon \Phi \longrightarrow \N \cup \{\infty\}
    \]
    is defined by
    \[
 P \cap U_{-\gamma}= u_{-\gamma} ({\boldsymbol{\alpha}}_{p^{\varphi(\gamma)}}), \quad \gamma \in \Phi^+.
\]
\end{definition}

In other words, any positive root $\gamma$ not belonging to the root system of the Levi subgroup $P_{\text{red}}\cap P_{\text{red}}^-$ is mapped to the integer corresponding to the height of the finite unipotent subgroup $P \cap U_{-\gamma}$; all other roots are sent to infinity. For instance, the function associated to the parabolic ${_mG} P^\alpha$ sends all positive roots to infinity, except those whose support contains $\alpha$, which are assigned the value $m$. The following fundamental structure result is \cite[Theorem 10]{Wenzel}.

\begin{theorem}
\label{thm:numericalfunction}
The parabolic subgroup $P$ is uniquely determined by the function $\varphi$, with no assumption on the characteristic or on the Dynkin diagram of $G$.
\end{theorem}

Before going into the complete classification, let us introduce some terminology and notation concerning the exotic objects appearing in characteristic $2$ and $3$.


\subsection{The very special isogeny}
Assume that the Dynkin diagram of \( G \) has an edge of multiplicity \( p \); that is, either \( G \) is of type \( B_n \), \( C_n \), or \( F_4 \) in characteristic \( p=2 \), or of type \( G_2 \) in characteristic \( p=3 \). Under these assumptions, one can define the \emph{very special isogeny} of \( \widetilde{G} \) as the quotient
\[
\pi_{\widetilde{G}} \colon \widetilde{G} \longrightarrow \overline{\widetilde{G}}
\]
by the subgroup \( N = N_{\widetilde{G}} \), which is normal, non-central, killed by Frobenius, and minimal with these properties. This subgroup is uniquely determined by its Lie algebra, defined as the smallest \( p \)-Lie subalgebra containing the root subspaces associated to all \emph{short} roots. More precisely, the isogeny \( \pi = \pi_{\widetilde{G}} \) acts as Frobenius on the copies of the additive groups corresponding to short roots, and as an isomorphism on the others.

\medskip

The quotient \( \overline{\widetilde{G}} \) remains simple and simply connected, and its Dynkin diagram is dual to that of \( G \). We denote by \( \overline{G} \) the corresponding adjoint group. Moreover, the composition
\[
\pi_{\widetilde{\overline{G}}} \circ \pi_{\widetilde{G}}
\]
equals the Frobenius morphism of \( \widetilde{G} \): this is shown  in the original works by Borel and Tits \cite{BorelTits} and also in \cite[Chapter 7]{CGP15} in more detail. We now extend the notion of very special isogeny to \( G \) as follows. When defined for \( \widetilde{G} \), its image under \( \sigma \colon \widetilde{G} \to G \) is denoted
\[
N_G = \sigma(N_{\widetilde{G}}),
\]
and similarly, for any \( m \geq 1 \), we define \( {}^m N_G \) as the quotient of \( {}^m N_{\widetilde{G}} \) by the center. Note that \( N_G \) remains non-central, normal, killed by Frobenius, and minimal with these properties. We denote by \( \pi_G \) the composition of the quotient of \( G \) by \( N_G \), followed by the quotient by the center, so that the following diagram commutes.
\begin{center}
    \begin{tikzcd}
        \widetilde{G} \arrow[rr, "\pi_{\widetilde{G}}"] \arrow[d, "\sigma"] && \widetilde{\overline{G}} \arrow[d, "\overline{\sigma}"]\\
        G \arrow[rr, "\pi = \pi_G"] && \overline{G}.
    \end{tikzcd}
\end{center}


\subsection{Two exotic parabolic subgroups} 
\label{exotic}
Let us assume that $G$ is a group of type $G_2$ and that $p=2$. For its Dynkin diagram, we adopt the convention from \cite{Bourbaki}; thah is, we denote as $\alpha_1$ the short simple root and as $\alpha_2$ the long one. In \cite{Maccan}, it is shown that there are exactly two $p$-Lie subalgebras of $\Lie G$ that contain strictly $\Lie P^{\alpha_1}$, namely:
\[\mathfrak{q}_1 := \Lie P^{\alpha_1} \oplus \mathfrak{g}_{-2\alpha_1-\alpha_2}, \quad \mathfrak{q}_2 := \Lie P^{\alpha_1}\oplus \mathfrak{g}_{-\alpha_1}\oplus \mathfrak{g}_{-\alpha_1-\alpha_2}.
\]
Moreover, lifting these two subalgebras to subgroups of height one (see \Cref{DG height one} for the precise statement) and then multiplying them by $P^{\alpha_1}$, one obtains two exotic parabolic subgroups $\mathcal{Q}_1$ and $\mathcal{Q}_2$, satisfying 
\[
\Lie \mathcal{Q}_1 = \mathfrak{q}_1, \quad \Lie \mathcal{Q}_2 = \mathfrak{q}_2 \quad \text{and} \quad  (\mathcal{Q}_1)_{\text{red}} = (\mathcal{Q}_2)_{\text{red}} = P^{\alpha_1}\]
and whose unipotent infinitesimal part has height one. As we recall below (Theorem \ref{th: classification parabolic subgroups}), these two parabolic subgroups are sufficient to complete the classification in characteristic two, together with the ones obtaining from reduced ones, the Frobenius homomorphism and the very special isogenies. For more details on the construction of $\mathcal{Q}_i$ and the corresponding homogeneous varieties, see \cite{Maccan}.


\subsection{The complete classification}
The following result \cite[Proposition 3, Proposition 4]{Maccan} addresses parabolic subgroups whose reduced part is maximal, with respect to inclusion among reduced parabolics.

\begin{theorem}
\label{th: classification parabolic subgroups}
Let $G$ be a simply-connected simple group.
    Let $P$ be a parabolic subgroup of $G$ such that $P_{\text{red}} = P^\alpha$ for some simple root $\alpha$. Then if $p \geq 5$, or if $G$ is simply laced, or if $p=3$ and $G$ is of type $B_n$, $C_n$ or $F_4$, there is an integer $r\geq 0$ such that \[
    P = {}_rG P^\alpha.\]
    Otherwise, the following cases arise.
    \begin{itemize}
        \item if $G$ is of type $B_n$, $C_n$ or $F_4$ and $p=2$, or if $G$ is of type $G_2$ and $p=3$, there is some $r\geq 0$ such that
        \[
        P = {}_mG P^\alpha \quad \text{or} \quad P = {}^mN P^\alpha,
        \]
        where ${}^mN = {}^mN_G$ denotes the kernel of the very special isogeny $\pi_G\colon G \to \overline{G}$, followed by an $m$-th iterated Frobenius morphism;
        \item if $G$ is of type $G_2$, $p= 2$ and $\alpha = \alpha_2$ is the long simple root, then there is some $m \geq 0$ such that $P = {}_mG P^{\alpha_2}$;
        \item if $G$ is of type $G_2$, $p=2$ and $\alpha = \alpha_1$ is the short simple root, then $P$ is obtained by pullback via $F^m$ from one among $P^{\alpha_1}$, $\mathcal{Q}_1$ and $\mathcal{Q}_2$.
    \end{itemize}
\end{theorem}

The following result completes the classification and is proven in \cite[Theorem 1]{Maccan2}.

\begin{theorem}
\label{thm: classification higher}
    Let $P \subset G$ be a parabolic subgroup of a simply-connected semisimple group $G$, and let $\beta_1,\ldots, \beta_r$ be the simple roots of $G$ such that $P_{\red}$ is the intersection of the parabolic subgroups $P^{\beta_i}$. 
    Then
    \[
    P = \bigcap_{i=1}^r Q^i,
    \] where $Q^i$ is the smallest subgroup of $G$ containing both $P$ and $P^{\beta_i}$. In particular, every parabolic subgroup $P$ can be expressed as the intersection of parabolic subgroups whose reduced part is maximal, independently of type and characteristic.
\end{theorem}

\subsection{Contractions of Schubert curves}
\label{sec contractions}
We start by recalling a fundamental result in equivariant birational geometry; see \cite[$7.2$]{Brion17}.

\begin{theorem}[Blanchard's Lemma]
\label{blanchard}
Let $f \colon X \rightarrow Y$ be a contraction between projective varieties over $k$. Assume 
that $X$ is equipped with an action of a connected algebraic group $G$. Then there exists a unique $G$-action on $Y$ such that the morphism $f$ is $G$-equivariant.
\end{theorem}

Next, consider a homogeneous variety
\[
X= G/P, \quad \text{with } P_{\text{red}} = P_I
\] 
and denote as $o$ its base point. We can define the Schubert divisors of $X$, associated to the simple roots not in the Levi subgroup of $P_{\text{red}}$, as the closures of the following orbits:
\begin{align}
\label{Dalpha}
D_\alpha\defeq \overline{B^- s_\alpha o}, \quad \alpha \in \Delta \backslash I.
\end{align}
The Picard group of $X$ is freely generated by the linear equivalence classes of the $D_\alpha$s. Moreover, a divisor is numerically effective if and only if with respect to such a basis its coefficients are all non-negative. In particular, the Picard rank of $X$ equals the cardinality of $\Delta \backslash I$.
In \cite[Remark 3.18]{Maccan}, we build a finite family of morphisms
\begin{align}
\label{f_alpha}
f_\alpha \colon X  \longrightarrow Y_\alpha \defeq \Proj \left(\bigoplus_{m\geq 0} H^0(X, \mathcal{O}_X(mD_\alpha))\right), \quad \alpha \in \Delta \backslash I.
\end{align}
We construct $f_\alpha$ as the the unique contraction on $G/P$ such that the Schubert curves 
\[
C_\beta \defeq \overline{U_{-\beta} o}, \quad \beta \in \Delta \backslash I,
\]
which are smooth, are all contracted to a point, except for $C_\alpha$. More precisely, the only curves that are not contracted to a point are those that are numerically proportional to $C_\alpha$. The Schubert curves are dual objects to the Schubert divisors, in the sense that they form a basis for the cone of effective $1$-cycles in $X$ up to numerical equivalence, and that moreover their intersection numbers satisfy $D_\alpha \cdot C_\beta = \delta_{\alpha\beta}$. As an application of \Cref{blanchard} to the map $f_\alpha$, the target must be of the form 
\[
Y_\alpha = G/Q^\alpha,\]
for some parabolic subgroup $Q^\alpha$ containing $P$. Then \cite[Lemma 3.19]{Maccan} proves that $Q^\alpha$ is in fact the subgroup generated by $P$ and $P^\alpha$.

\begin{remark}
\label{minimale}
    As a consequence of \Cref{th: classification parabolic subgroups}, in all types except for $G_2$ in characteristic $2$, there is a unique isogeny $\xi_\alpha$ with no central factor, such that 
    \[
Q^\alpha = (\ker \xi_\alpha) P^\alpha.
\]
By isogenies with no central factors we mean those isogenies of the form
\[
F^m_G \quad \text{or} \quad F^m_{\overline{G}} \circ \pi_G,
\]
for some $m\geq 0$, where the second case only occurs when the Dynkin diagram of $G$ has an edge of multiplicity $p$ equal to the characteristic. Kernels of such isogenies are totally ordered by inclusion:
\begin{align}
\label{nocentral}
1 \subsetneq N \subsetneq {}_1G \subsetneq {}^1N \subsetneq \ldots \subsetneq {}_mG \subsetneq {}^mN \subsetneq {}_{m+1}G \subsetneq \ldots,
\end{align}
where ${}^mN_G$ (which we denote simply by ${}^mN$ when $G$ is implicit) is the kernel of the composition of a very special isogeny and an $m$-th iterated Frobenius morphism. Moreover, any isogeny with source $G$ is of the form (\ref{nocentral}) followed by a finite map with central kernel: \cite[Proposition 2]{Maccan}.
\end{remark}

\begin{remark}
\label{rem: psmooth}
    Let $X=G/P$ and assume that 
    $P$ does not contain the kernel of any isogeny with no central factor. Assume moreover that $P$ is not among the two exotic parabolic subgroups $\mathcal{Q}_1$ and $\mathcal{Q}_2$ in type $G_2$ and characteristic two. Then there is a unique reduced parabolic subgroup of $G$, minimal among those containing $P$, given by
    \[
    P^{\sm} =P_J = \bigcap \{P^\alpha \colon \alpha \in \Delta \backslash I \, \text{ and } \, Q^\alpha=P^\alpha\}.
    \]
    Moreover, $P$ can be written in a unique way as 
    \begin{align}
    \label{intersezione}
    P = P_J \cap (\ker \xi) P^\prime,
    \end{align}
    where $\xi$ is an isogeny with no central factor and $P^\prime$ is a parabolic subgroup with reduced part
    \[
    P^\prime_{\text{red}} = \bigcap \{P^\alpha \colon \alpha \in \Delta \backslash I \, \text{ and } \, Q^\alpha \neq P^\alpha\},
    \]
    In order for this expression to be unique we moreover have to assume that $\xi$ is minimal with respect to inclusion. To illustrate this, let us consider $G$ to be of type $G_2$ and assume $p=3$, so that we can define the very special isogeny with kernel $N$. Then 
    \[
    P^{\alpha_2} \cap {}_1G P^{\alpha_1} = P^{\alpha_2} \cap NP^{\alpha_1}.
    \]
    This comes from the fact that there is no long root $\gamma$ such that $\alpha_1 \in \Supp(\gamma)$ and $\alpha_2 \notin \Supp(\gamma)$. The only root satisfying both conditions is $\alpha_1$ itself, which is short.
\end{remark}


\section{Demazure's exceptional pairs}
\label{sec Demazure}

Let us make precise what we mean by automorphism group: for a projective variety $X$, the functor 
\[
\underline{\Aut}_X \colon (\mathbf{Sch}/k) \longrightarrow \mathbf{Grp}, \quad T \longmapsto \Aut_T(X_T),
\]
sending a $k$-scheme $T$ to the group of automorphisms of $T$-schemes of $X\times_k T$, is represented by a locally algebraic group $\underline{\Aut}_X$ \cite[Theorem 3.6]{MatsumuraOort}. 
We denote \(\underline{\Aut}_X^0\) as its identity component, which is a connected algebraic group.

\medskip

If \(X = G/P\) with $G$ simple adjoint, then the reduced subgroup \((\underline{\Aut}^0_X)_{\text{red}}\) is a simple group of adjoint type.  Moreover, if \(P\) is a reduced parabolic subgroup, then $\underline{\Aut}_X^0$ is known to coincide with its reduced part. More precisely, the following result of Demazure completely describes the automorphism group of $X$ when the stabilizer is a reduced parabolic \cite[Theorem 1]{Demazure}.

\begin{theorem}
\label{demazure77}
Let $G$ be a semisimple adjoint group over $k$ and $P$ a reduced parabolic subgroup of $G$. Then the natural homomorphism
\[
G \longrightarrow  \underline{\Aut}^0_{G/P}
\]
is an isomorphism in all but the three following cases: 
\begin{enumerate}[(1)]
    \item $G$ is of type $C_n$ for some $n \geq 2$ and $P= P^{\alpha_1}$ is associated to the first short simple root: in this case the automorphism group $\hat{G}$ is simple adjoint of type $A_{2n-1}$; 
    \item $G$ is of type $B_n$ for some $n \geq 2$ and $P = P^{\alpha_n}$ is associated to the short simple root: in this case the automorphism group $\hat{G}$ is simple adjoint of type $D_{n+1}$; 
    \item $G$ is of type $G_2$ and $P = P^{\alpha_1}$: in this case the automorphism group $\hat{G}= \SO_7$ is simple adjoint of type $B_3$.
\end{enumerate}
\end{theorem}

With a slight change of notation from Demazure, the three pairs $(G,\alpha)$ appearing in cases $(1)$, $(2)$, and $(3)$ of \Cref{demazure77} are called \emph{exceptional}, and their associated automorphism group is denoted by $\hat{G} \supset G$. The corresponding homogeneous varieties (which we also call exceptional, when such a terminology arises no ambiguity) all have Picard rank one.

\medskip

Let us now proceed with a closer examination of these three families of exceptional pairs $(G,\alpha)$, with some useful facts concerning their Lie algebras and subalgebras. Particular care is required in small characteristics (two and three), where the kernel of the very special isogeny plays a significant role.

\begin{remark}
\label{description BD}
    If $G = \SO_{2n+1} \subset \hat{G}$, it is easier to perform the computation with the embedding $G \subset \hat{G} = \PSO_{2n+2}$. We can realize $G$ as the stabilizer of a non-isotropic line, as follows. Let us fix some coordinates: the quadratic form on $k^{2n+2}$, with canonical basis $e_0,\ldots, e_{2n+1}$, is given by
\[
x_0^2 + x_0x_{2n+1} + x_{2n+1}^2 + \sum_{i = 1}^{n-1} x_i x_{2n+1-i},
\]
and the vector we choose is $v_0 := e_0 + e_{2n+1}$. This way, the stabilizer of $k v_0$ is isomorphic to $G$, acting on
\[
v_0^\perp = k^{2n+1} = kv_0 \oplus ke_1 \oplus \ldots \oplus k e_{2n}
\]
preserving the quadratic form
\[
y_0^2 + \sum_{i=1}^{n-1} y_i y_{2n+1-i}.
\]
This has the advantage of being a description that works in any characteristic, including $p=2$. 
\end{remark}

\begin{lemma}
\label{keylemmaG}
    Let $(G,\alpha)$ be exceptional and denote as $\hat{G}$ the automorphism group of $G/P^\alpha$. Assume $\mathfrak{k} \subsetneq \Lie \hat{G}$ is both a Lie subalgebra and a $G$-submodule, satisfying $\mathfrak{k} \cap \Lie G = 0$. Then $\mathfrak{k}=0$ or $\mathfrak{k} = \mathfrak{z}(\Lie \hat{G})$ is the one-dimensional center of the Lie algebra. The latter can only happen for $p=2$ and a group $G$ of type $B_n$ or $G_2$. 
\end{lemma}

\begin{proof}
  Let us assume that $\mathfrak{k}$ is nonzero. Then the assumption $\Lie G \cap \mathfrak{k} = 0$ means that we can see $\mathfrak{k}$ as a nontrivial $G$-submodule of 
  \[
  W := \Lie \hat{G}/\Lie G.
  \]
  We now proceed with each exceptional pair separately.

  \smallskip
  
  If $G = \PSp_{2n} \subset \hat{G} = \PGL_{2n}$, then independently of the characteristic we have $W \simeq \Lambda^2 k^{2n}$, which is a simple $G$-module because the Weyl group of $G$ acts transitively on the set of its weights. So $\mathfrak{k} \neq 0$ implies that $\mathfrak{k} \simeq \Lambda^2 k^{2n}$ as $G$-submodules of $\Lie \hat{G}$. If $p=2$ such a submodule does not exist, because $\Lie G$ does not admit a $G$-stable complementary subspace inside $\Lie \hat{G}$. If $p \geq 3$, then such a subspace indeed exists (one can write any traceless matrix as the sum of a symmetric and an antisymmetric one); however it is not stable by the bracket of $\Lie \hat{G}$. Indeed, the bracket of two skew-symmetric traceless matrices can pick up a component in the complementary subspace of symmetric matrices.

\smallskip 

If $G = \SO_{2n+1} \subset \hat{G} = \PSO_{2n+2}$, then 
\[
\Lie \hat{G} = \Lambda^2 k^{2n+2}, \quad \Lie G = \Lambda^2 k^{2n}, \quad W \simeq k^{2n+1},
\]
where the $G$-action on $W$ is the one given by the standard representation. If $p \geq 3$, then $W$ is a simple $G$-module thus $\mathfrak{k}\neq 0$ implies that $\mathfrak{k} \simeq k^{2n+1}$ is a $G$-stable complementary subspace of $\Lie G $ inside $\Lie \hat{G}$. Such a subspace exists, however it is not stable by the bracket of $\Lie \hat{G}$, so we get a contradiction. If $p=2$, then $W$ is not simple; more precisely, it admits exactly one nontrivial submodule $l \subset W$ on which $G$ acts trivially. If $\mathfrak{k} = W$ we get a contradiction because the splitting does not exist in characteristic two, so we must have $\mathfrak{k} \simeq l$ as $G$-submodules of $\Lie \hat{G}$. Then $\mathfrak{k}$ has to coincide with the only line in $\Lie \hat{G}$ on which $G$ acts trivially; namely it is generated by $e_0 \wedge e_{2n+1}$ and it coincides with the center $\mathfrak{z}(\Lie \hat{G})$, so it is also stable by the bracket.

  \smallskip

If $G \subset \hat{G} = \SO_7$ is an exceptional group of type $G_2$, then $W \simeq k^7$ is the $G$-module associated to the standard representation of $G_2$ as acting on the hyperplane of pure octonions. If $p \geq 3$, then $W$ is a simple $G$-module and thus $\mathfrak{k} \simeq k^7 \subset \Lie \hat{G}$. There exist such a $G$-stable complementary subspace, but analogously as in the previous cases, it is not stable by the bracket of $\Lie \hat{G}$. If $p=2$, then $W$ is not simple and it admits exactly a one-dimensional submodule $l \subset W$ on which $G$ acts trivially. If $\mathfrak{k} \simeq W$ we get a contradiction because the splitting does not exist in characteristic $2$, so we must have $\mathfrak{k} \simeq l$. Again, this means that $\mathfrak{k}$ is the only nontrivial subspace of $\Lie \SO_7$ that is fixed by $G_2$, which is the one-dimensional center $\mathfrak{z}(\Lie \hat{G})$, and we are done.
\end{proof}

\begin{lemma}
\label{keylemmaN}
    Assume that the very special isogeny of $G$ is defined. Let $(G,\alpha)$ be exceptional and denote as $\hat{G}$ the automorphism group of $G/P^\alpha$. Assume $\mathfrak{k} \subsetneq \Lie \hat{G}$ is both a Lie subalgebra and a $G$-submodule and that moreover $\mathfrak{k} \cap \Lie G = \Lie N$.
    
    Then necessarily $\mathfrak{k} = \Lie N$, unless $\mathfrak{k} = \Lie N \oplus \mathfrak{z}(\Lie \hat{G})$ when $G$ is of type $B_n$ and $p=2$.
 \end{lemma}

\begin{proof}
Let us assume that $\mathfrak{k}$ strictly contains $\Lie N$. By the assumption $\Lie G\cap \mathfrak{k} = 0$, the projection $\mathfrak{p}$ of $\mathfrak{k}$ onto the quotient $W = \Lie \hat{G}/\Lie G$ is a non-trivial $G$-module.

 If $p=2$ and $G = \PSp_{2n}$, then $W = \Lambda^2 k^{2n}$ is a simple $G$-module, hence it must coincide with $\mathfrak{p}$. However, $W$ does not lift to a $G$-submodule inside $\Lie \hat{G}$, so we have a contradiction.

If $p=2$ and $G = \SO_{2n+1}$, by reasoning in the analogous way as in the proof of \Cref{keylemmaG}, we find that either $\mathfrak{p} = W$ or that $\mathfrak{p} = k(e_0 \wedge e_{2n+1)}$ is the projection of the center of $\Lie \hat{G}$. Thus as a subspace either $\mathfrak{k} = \Lie N \oplus k^{2n+1}$, which is not stable by the $G$-action, or $\mathfrak{k} = \Lie N \oplus \mathfrak{z}(\Lie \hat{G})$. 

    If $p=3$ and $G$ is an exceptional group of type $G_2$, then $\Lie \hat{G} = \Lie G \oplus W$ decomposes as a direct sum, and moreover $W$ is simple as a $G$-module. Thus, the condition $\mathfrak{k} \neq \Lie N$ implies that $\mathfrak{k} = \Lie N \oplus W$. Since both $\Lie N$ and $W$ are $G$-modules, they are in particular stable by taking brackets with $\Lie G$, i.e.
    \[
    [\Lie N, \Lie G] \subset \Lie N \subset \mathfrak{k} \quad \text{and} \quad [W, \Lie G ]\subset W \subset \mathfrak{k}.
    \]
    This implies that $[\mathfrak{k}, \Lie G] \subset \mathfrak{k}$. If we assume that $\mathfrak{k}$ is a Lie subalgebra of $\Lie \hat{G}$, then we get that $[\mathfrak{k}, \Lie \hat{G}] \subset \mathfrak{k}$, because $\Lie \hat{G} = \Lie G + \mathfrak{k}$. In other words, $\mathfrak{k}$ is a nontrivial Lie ideal of $\Lie \hat{G}$. However this cannot happen because $\Lie \SO_7$ is simple as a Lie algebra in characteristic $p=3$; indeed, a Lie algebra of type $B_3$ is simple in any characteristic except for $p=2$.
\end{proof}

To obtain more detailed information, it is necessary to go beyond the Lie algebra level and examine the corresponding homogeneous varieties of Picard rank two more closely.

\begin{proposition}
    \label{no mu2 case1}
    Let $p=2$ and $X=G/P$ with $G$ of type $B_n$ and
    \[
    P=P^{\alpha_n} \cap {}_m G P^{\alpha_i}
    \]
    for some $m \geq 1$ and some $1\leq i <n$. Let $H \subset \hat{G} = \PSO_{2n+2}$ be a copy of $\boldsymbol{\mu}_{p^{m+1}}$ such that $\Lie H = \mathfrak{z}(\Lie \hat{G})$.
    
    Then the action of $H$ on the product $G/P^{\alpha_n} \times G/{}_mGP^{\alpha_i}$ (standard twisted by the $m$-th Frobenius on the first factor and trivial on the second) does not preserve the variety $X$.
\end{proposition}

\begin{proof}
We keep the notation in coordinates given in \Cref{description BD}. The variety $X$ is the incidence variety of pairs of the form $(E,W)$, where $E \subset k^{2n+1}$ is a totally isotropic subspace of dimension $i$ and $W \subset k^{2n+1}$ is a totally isotropic subspace of dimension $n$. The incidence relation is twisted by the $m$-th iterated Frobenius; that is, $E\subset F^m(W)$. Let us recall that the quotient $G/P^{\alpha_n}$ is isomorphic to the isotropic Grassmannian of $(n+1)$-dimensional totally isotropic subspaces of $k^{2n+2}$. An element $t \in H$ acts on such a subspace by multiplying by $t$ the $0$-th coordinate and by $t^{-1}$ the $(2n+1)$-th coordinate. Restricting to
\[
k^{2n+1} := (e_0+e_{2n+1})^\perp = v_0^\perp,
\]
this action becomes a scaling of the $0$-th coordinate alone. Now, let us consider as base point of the variety $X$ the one given by
    \begin{align*}
    E_0 & = \{ x \in k^{2n+1} \colon x_0 = x_1 = x_{2n}, \, x_{i+1} = \ldots = x_{2n-1} = 0\},\\
    W_0 & = \{ y \in k^{2n+1} \colon y_0 = y_1 = y_{2n}, \, y_{n+1} = \ldots = y_{2n-1} = 0\}.
    \end{align*}
    Consider a generator $t \in \boldsymbol{\mu}_{p^{m+1}}$, so that $s := t^{p^m}$ is a nontrivial element of $ \boldsymbol{\mu}_2$. Thus, 
    \[
    F^m(t \cdot W_0) = \{(sz,z,z_1,\ldots,z_n,0, \ldots, 0, z)\}
    \]
    which does not contain $E_0$ anymore because its first two coordinates are not equal to each other.
\end{proof}

\begin{proposition}
    \label{no mu2 case2}
    Let $p=2$ and $X=G/P$ with $G$ a group of type $G_2$ and 
    \[
    P= P^{\alpha_1} \cap {}_mGP^{\alpha_2}
    \]
    for some $m \geq 1$. Let $H \subset \hat{G} =  \SO_7$ be a copy of $\boldsymbol{\mu}_{p^{m+1}}$ such that $\Lie H = \mathfrak{z}(\Lie \hat{G})$.
    
    Then the action of  $H$ on the product $G/P^{\alpha_1} \times G/{}_mGP^{\alpha_2}$ (standard twisted by the $m$-th Frobenius on the first factor and trivial on the second) does not preserve the variety $X$.
\end{proposition}

\begin{proof}
The proof goes through in the analogous way as the one of \Cref{no mu2 case1}. Here, the parabolic subgroup $P^{\alpha_1}$ is the stabilizer of a line in the seven-dimensional vector space $V$ of pure octonions, thus the quotient $G/P^{\alpha_1}$ is simply the isotropic Grassmannian of lines under the $\SO_7$-action. The incidence relation here is of the form $F^m(l) \subset E$, where $l \subset V$ is one-dimensional and $E \subset V$ is two-dimensional, and on both of them the product of octonions is identically zero. In order to perform the computation, let us fix some coordinates; more details on the Chevalley embedding of $G_2$ can be found in \cite[Appendix, Section 4]{Maccan}. Let us summarize what we need from such a description: we can arrange $V$ using the standard basis of the octonions
\[
V = kf_{12} \oplus kf_{11} \oplus ke_{12} \oplus k(e_{11}+e_{22}) \oplus ke_{21} \oplus kf_{22} \oplus kf_{21}.
\]
This way, we choose the maximal torus of $G$ and $\hat{G}$ to be respectively
\[
\Gm^2 \simeq T = \diag(a,a^{-1}b^{-1}, a^2b, 1,a^{-2}b^{-1},ab,a^{-1}) \subset \hat{G} = \diag(t_1,t_2,t_3,1,t_3^{-1}, t_2^{-1},t_1^{-1}).
\]
A basis of simple roots of $G$ is given by $\alpha_1(a,b) = a$ and $\alpha_2(a,b) = b$, while a basis of simple roots for $\hat{G}$ is given by $\beta_1(t) = t_1t_2^{-1}$, $\beta_2(t) = t_2t_3^{-1}$ and $\beta_3(t) = t_3$. With a slight change of notation from \Cref{description BD}, we set the quadratic form of $\SO_7$ as being $x_3^2 + x_2x_4 + x_1x_5 + x_0x_6$. As base points for the homogeneous varieties that we are considering, we can take
\[
l_0 = \{ (0,0,x,x,x,0,0)\colon x \in k\} = k (e_{11} +e_{22}+ e_{12} + e_{21}),
\]
and by a direct computation one sees that its generator has square zero with respect to the product of octonions. As a two dimensional vector space, we can take
\[
E_0 = \{ (z,z,y,y,y,z,z) \colon y,z \in k\} = k(e_{11}+e_{22}+ e_{12}+ e_{21})\oplus k(f_{11}+f_{22}+f_{21}+f_{12}),
\]
on which again one may check that the product of octonions is zero. Now, an element $t \in H$ acts via the simple root $\beta_3$ and a Frobenius twist meaning that
\[
F^m(t \cdot l_0) = \{ (0,0,sx,x,sx,0,0) \colon x \in k \}, \quad s = t^{p^m} \in \boldsymbol{\mu}_2.
\]
In particular, since the action of $H$ is assumed to be trivial on the subspace $E_0$, the incidence relation is not preserved by the action of $H$.
\end{proof}

\begin{proposition}
    \label{no mu2 case3}
    Let $p=2$ and $X=G/P$ with $G$ of type $B_n$ and 
    \[
    P = G/P^{\alpha_n} \cap {}^mNP^{\alpha_i},
    \]
    for some $m \geq 1$ and some $1\leq i <n$. Let $H \subset \hat{G} = \PSO_{2n+2}$ be a copy of $\boldsymbol{\mu}_{p^{m+1}}$ such that $\Lie H = \mathfrak{z}(\Lie \hat{G}).$ Then the action of $H$ on the product $G/P^{\alpha_n} \times G/{}^m NP^{\alpha_i}$ (standard twisted by the $m$-Frobenius on the first factor and trivial on the second) does not preserve the variety $X$.
\end{proposition}

\begin{proof}
Here we perform an analogous computation to the one in the proof of \Cref{no mu2 case1}, with a slight modification.Recall that we see $G \subset \PSO_{2n+2}$ as the stabilizer of the non-isotropic line $kv_0$. Moreover, the first factor $G/P^{\alpha_n}$ parametrized the $(n+1)$-dimensional totally isotropic subspaces $\widetilde{W} \subset k^{2n+2}$, which correspond to $n$-dimensional totally isotropic subspaces in $k^{2n+1} = v_0^\perp$ by the bijection
\[
W = \widetilde{W}\cap v_0^\perp.
\]
On the other hand, the second factor $Y= G/ {}^mN P^{\alpha_i}$ is isomorphic to the Lagrangian grassmannian of subspaces of dimension $i$, inside the symplectic quotient $k^{2n}$. In order to work in the ambient $k^{2n+2}$ on which the $H$-action is defined, let us see $Y$ as the space parametrizing all those $(i+1)$-dimensional subspaces $E \subset v_0^\perp$ which contain $v_0$ and are moreover isotropic with respect to the quadratic form. The incidence relation defining the subvariety $X$ is thus
\[
E \subset F^m(\widetilde{W}) \cap v_0^\perp.
\]
Considering once again the central $H = \boldsymbol{\mu}_{2^{m+1}}$ with generator the variable $t$, let us choose as base point the one given by
\[
\widetilde{W}_0 :=  ke_1\oplus \cdots \oplus ke_{n+1} \subset k^{2n+2} \quad \text{and} \quad E_0 := kv_0 \oplus ke_1\oplus \cdot \oplus ke_i \subset v_0^\perp.
\]
By performing the same computation as in \Cref{no mu2 case1}, we see that $E_0$ is not contained in $F^m(t\cdot \widetilde{W}_0) \cap v_0^\perp$, which means that the incidence relation is not preserved by the $H$-action.
\end{proof}

\begin{lemma}
\label{KequalsN}
    Assume that the very special isogeny of $G$ is defined.  Let $(G,\alpha)$ be exceptional and denote as $\hat{G}$ the automorphism group of $G/P^\alpha$. Let $K$ be an infinitesimal subgroup of $\hat{G}$ that is normalized by $G$ and that has Lie algebra $\Lie K = \Lie N$. Then $K = N$.
\end{lemma}

\begin{proof}
It suffices to show that the connected component of the normalizer satisfies $N_{\hat{G}}(\Lie N)^0 = G$. Indeed, this implies
    \[
    K \subset N_{\hat{G}}(K)^0 \subset N_{\hat{G}}(\Lie K )^0 = N_{\hat{G}}(\Lie N)^0 = G,
    \]
    so that $K$ is an infinitesimal normal subgroup of $G$ with Lie algebra equal to $\Lie N$. By the factorisation of isogenies (\Cref{minimale}), this forces $K = N$ and we are done. So let us analyze the three exceptional situation separately.

\smallskip

    \textsf{First case.} Let $p=2$ and consider $G = \PSp_{2n} \subset \hat{G} = \PGL_{2n}$ and define $S := N_{\hat{G}}(\Lie N)^0$. We know that $S$ contains $G$ so it suffices to prove that this inclusion is an equality. Let us assume $S \supsetneq G$, then the Lie algebra of $S$ should be a $G$-module satisfying
    \[
    \Lie G \subsetneq \Lie S \subset \Lie \PGL_{2n}.
    \]
    Once again, the quotient $\Lie \hat{G}/\Lie G$ being an irreducible $G$-module, we have that this implies $\Lie S = \Lie \hat{G}$. But this would imply that $\Lie N$ is a Lie ideal of the full Lie algebra $\Lie \hat{G}$, which cannot be the case because this Lie algebra is simple. Thus $\Lie S = \Lie G$ which forces $S=G$.

    \smallskip

\textsf{Second case.} Let $p=2$ and $G = \SO_{2n+1}$, and let us keep the notation of \Cref{description BD}. By an analogous computation as the one of \cite[Example 1.15]{Maccan}, we can identify
 \begin{align*}
     \Lie N & = \{a_{ij} = 0, \, \text{except for } i \in \{0,2n+1\} \text{ and } 1 \leq j \leq 2n, \, a_{0,j} = a_{2n+1,j} \}\\
     & \simeq kv_0 \wedge \left( \oplus_{j = 1}^{2n} ke_j \right) \simeq k^{2n}
 \end{align*}
 where we make the identification $\Lie \hat{G} = \Lie \SO_{2n+2} = \Lambda^2 (k^{2n+2}).$ As in the previous case, in order to conclude it is enough to show that 
 \[
 N_{\SO_{2n+2}} (\Lie N)^0 \subset \SO_{2n+1}.
 \]
 Now, let $g= (m_{ij})_{i,j=0}^{2n+1}$ be an element of $\SO_{2n+2}$ and we assume that $g$ stabilizes $\Lie N$. Then it suffices to show that $g \cdot v_0$ is a scalar multiple of $v_0$, that is
 \begin{align*}
     & m_{0,0} + m_{0,2n+1} = m_{2n+1,0} + m_{2n+1,2n+1}; & \\
     & m_{k,0} + m_{k,2n+1} = 0, & \text{ for all } 1 \leq k \leq 2n.
 \end{align*}
A direct computation by multiplying the corresponding vectors by $g$ shows that $g \cdot (v_0 \wedge e_j) = g\cdot v_0 \wedge g \cdot e_j$ has as coefficient in front of $e_k \wedge e_l$ exactly
\[
(m_{k,0} + m_{k,2n+1})m_{l,j} + (m_{l,0} + m_{l,2n+1}) m_{kj},
\]
where $1 \leq k,j,l \leq 2n$ and $k \neq l$. Since we are asking $g$ to stabilize $\Lie N$, all these coefficients must vanish, and setting $\lambda_k := m_{k,0} + m_{k,2n+1}$ we get
\[
\lambda_k m_{lj}= \lambda_l m_{kj}. 
\]
Now let us assume that there is some $k_0$ such that $\lambda_{k_0} \neq 0$, then setting $l = k_0$, for all $k \neq k_0$ and for any $j$ we have that
\[
m_{kj} = \frac{\lambda_k}{\lambda_{k_0}} m_{k_0,j}
\]
so that the $j$-th column $(m_{kj})_{k = 1}^{2n}$ is a scalar multiple of the vector $(\lambda_k)_k$. This holds for all $j$ and thus this implies that the submatrix $(m_{ij})_{i,j=1}^{2n}$ of $g$ has rank at most equal to one. This contradicts the fact that $g$ is invertible, hence we conclude that all $\lambda_k = m_{k,0} + m_{k,2n+1}$ must vanish. Next, for any $j$, the fact that $g \cdot (v_0 \wedge e_j) \in \Lie N$ implies that its coefficient with respect to $e_0 \wedge e_l$ and the one with respect to $e_{2n+1}\wedge e_l$ must be equal for all $1 \leq l \leq 2n$. By a direct computation we see that this means
\[
(m_{0,0} + m_{k,2n+1})m_{lj} + \lambda_l m_{0,j} = (m_{0,2n+1} + m_{2n+1,2n+1})m_{lj} + \lambda_l m_{2n+1,j}.
\]
and thus by choosing some $m_{lj} \neq 0$ and using the step above which says $\lambda_l = 0$ we get $m_{0,0} + m_{k,2n+1} = m_{0,2n+1} + m_{2n+1,2n+1}$ and we are done.

    
\smallskip

    \textsf{Third case.} Let $p=3$ and consider $G$ to be a group of type $G_2$ embedded into $\hat{G}= \SO_7$. As in the previous cases, it suffices to show that $S:= N_{\SO_7}(\Lie N)^0 = G$. Assume that the inclusion $G \subset S$ is instead a strict inclusion. Then $G$ acts on the Lie algebra of $S$, which satisfies
    \[
    \Lie G \subsetneq \Lie S \subset \Lie \SO_7 = \Lie G \oplus W,
    \]
    where $W \simeq V(\varpi_1)$ is an irreducible $G$-module. Thus the inclusion being strict implies that $\Lie S = \Lie SO_7$. However, this means that $\Lie N$ is an ideal inside the Lie algebra of $\SO_7$, which is a contradiction with the fact that such a Lie algebra is simple. So we conclude that $\Lie S = \Lie G$ and in particular that the stabilizer must coincide with $G$. 
\end{proof}


\section{The connected automorphism group}
\label{sec automorphism}
The goal of this section is to precisely characterize those among the non-reduced parabolics which make \(\underline{\Aut}_X^0\) non-reduced, as well as to determine the exact structure of the corresponding group scheme. We exclude the exotic parabolic subgroups in type $G_2$ and characteristic two.  

 \subsection{The reduced automorphism group} \Cref{demazure77} addresses all projective homogeneous varieties of Picard rank one, except for the unique exotic variety with stabilizer \(\mathcal{Q}_2\). If \(X\) is a projective homogeneous variety of Picard rank at least two under the action of a simple adjoint group, then by (\ref{intersezione}) it can be written as \(X = G/P\), where \(G\) is simple adjoint and
\[
P = P_J \cap (\ker \xi) P^\prime.
\]
Here, \(\xi\) denotes a non-central isogeny with source \(G\); specifically, \(\xi\) is either an iterated Frobenius morphism or the composition of an iterated Frobenius with a very special isogeny, when such an isogeny exists. Moreover, \(\xi\) is assumed to be minimal (with respect to inclusion) among all isogenies satisfying this expression for \(P\). This minimality assumption is useful for interpreting \(\ker \xi\) in terms of contractions of Schubert curves. For the moment, the exotic cases of Picard rank two in type \(G_2\) with \(p=2\) are excluded and will be treated separately in \Cref{G2 rank two}.

\begin{remark}
\label{rem: beta not exceptional}
Notice that if $(G,\alpha)$ is exceptional, then for any other simple root $\beta \neq \alpha$ one has that $(G,\beta)$ is not exceptional, which means \[
\underline{\Aut}^0_{G/{}_mGP^\beta} = \underline{\Aut}^0_{G/P^\beta}=G^{(m)}\]
for all $m \geq 0$. Moreover, if the very special isogeny of $G$ is defined, namely for $(\PSp_{2n},\alpha_1)$ or $(\SO_{2n+1},\alpha_n)$ when $p=2$ and for $(G_2,\alpha_1)$ when $p=3$, then again for any simple root $\beta \neq \alpha$ one has that $(\overline{G}, \overline{\beta})$ is not exceptional, which means that
\[
\underline{\Aut}^0_{G/{}^mNP^\beta} = \underline{\Aut}^0_{\overline{G}/P^{\overline{\beta}}}=\overline{G}^{(m)}
\]
for all $m\geq 0$. 
\end{remark}

\begin{lemma}
\label{lem: reduced aut}
Consider $X=G/P$ with $P = P_J \cap (\ker \xi)P^\prime$ a non-reduced parabolic subgroup of the above form, or let $X = G/\mathcal{Q}_2$ the exotic $G_2$-variety of Picard rank one. Then
\[
(\underline{\Aut}_X^0)_{\text{red}} = G.
\]
\end{lemma}

\begin{proof}
We first deal with the case of Picard rank at least two, and treat the case of $\mathcal{Q}_2$ later. Let $H = (\underline{\Aut}_X^0)_{\text{red}}$ and consider the contraction $f \colon X \to G/P_J$. First, let us assume that $P_J$ is not of the form $P^\alpha$ for any $\alpha$ simple root such that $(G,\alpha)$ is exceptional in the sense of Demazure. Then the map
\[
f_\ast \colon H \longrightarrow G = \underline{\Aut}^0_{G/P_J}
\]
is a section of the natural map $G \to H$ given by the $G$-action on $X$. In particular, both source and target being simple groups, this means that $f_\ast$ is an isogeny. The group $H$ then writes as the semi-direct product of $G$ and $\ker f_\ast$. Since $H$ is by construction smooth and connected, the finite subgroup $\ker f_\ast$ is both connected and étale, so it has to be trivial and we can conclude that $H=G$. Next, let us assume that $P_J = P^\alpha$ is exceptional. Then the composition
\[
G \longrightarrow H \stackrel{f_\ast}{\longrightarrow} \hat{G} = \underline{\Aut}^0_{G/P^\alpha}
\]
is the inclusion of $G$ into $\hat{G}$. In particular, the image $\hat{H}$ of $f_\ast$ is a simple group of adjoint type satisfying $G \subset \hat{H} \subset \hat{G}$. In all cases, namely $\PSp_{2n} \subset \PGL_{2n}$, $\SO_{2n+1} \subset \SO_{2n+2}$ and $G_2 \subset \SO_7$, the only possibilities are $\hat{H}= G$ and $\hat{H} = \hat{G}$. This is because $G$ is maximal with respect to inclusion among the semisimple subgroups of $\hat{G}$: \cite{seitz}. 
Reasoning in the same way as in the previous case, we find that $H \in \{G, \hat{G}\}$, so in order to conclude it is enough to show that $H = (\underline{\Aut}^0_X)_{\text{red}} \neq \hat{G}$.  Let us assume that $\hat{G}$ acts faithfully on $X$. Next, consider a simple root $\beta$ such that $\langle P,P^\beta\rangle = (\ker \xi) P^\beta$. Let us consider the associated contraction 
\[
f_\beta \colon X \longrightarrow G/(\ker \xi)P^\beta ;
\]
the target of $f_\beta$ is either of the form $G/{}_mGP^\beta$ for some $m\geq 1$ or, in the cases where the very special isogeny exists, to $G/{}^mN P^\beta$ for some $m\geq 0$. In both cases, by \Cref{rem: beta not exceptional} we know that its automorphism group is either $G^{(m)}$ or $\overline{G}^{(m)}$. The underlying varieties - that is, forgetting the structure of homogeneous spaces - are respectively isomorphic to $G/P^\beta$ or to $\overline{G}/P^{\overline{\beta}}$, where $\overline{\beta}$ is the associated simple root to $\beta$ via the very special isogeny $\pi_G \colon G \to \overline{G}$. By applying \Cref{blanchard} to $f_\beta$, we get
\begin{center}
    \begin{tikzcd}
        F^m_G \colon G \arrow[r]  & H = \hat{G} \arrow[rr, "(f_\beta)_\ast"] && G^{(m)}
    \end{tikzcd}
\end{center}
in the first case, and 
\begin{center}
    \begin{tikzcd}
        F^m_{\overline{G}}\circ \pi_G \colon G \arrow[r] & H = \hat{G} \arrow[rr, "(f_\beta)_\ast"] && \overline{G}^{(m)}
    \end{tikzcd}
\end{center}
in the second, where the first arrow is the natural map given by the $G$-action on $X$. In both cases, the homomorphism $(f_\beta)_\ast$ has to be faithfully flat, thus an isogeny (because $\hat{G}$ is simple and its action on the target is not trivial). However, $\hat{G}$ does not have the same dimension as $G$ and $\overline{G}$, which gives a contradiction. From this, we can conclude that $H = (\underline{\Aut}^0_X)_{\text{red}} $ has to be equal to $G$.

Finally, let $G$ be exceptional of type $G_2$ in characteristic $p=2$ and let us consider the variety $ Y = G/\mathcal{Q}_2$. In order to prove that its reduced automorphism group is of type $G_2$, it is enough to notice that $Y$ cannot be written as $G^\prime/P^\alpha$ for any semi-simple adjoint group $G^\prime$ and some simple root $\alpha$. This is proven in \cite[Proposition 2.38]{Maccan} and so the proof is complete.
\end{proof}

\subsection{Dimension of global vector fields}

The discussion above might suggest that the automorphism group is always reduced and therefore always equal to \(G\), but this is not necessarily true. One isolated example in the literature examines this phenomenon: \cite[Proposition 4.3.4]{BSU}. We briefly recall this example, providing a slightly different proof along with a generalization to other cases.

\begin{example}
    Let $G= \Sp_{2n}$ in characteristic $p \geq 3$ and consider $P = P^{\alpha_1} \cap \, {}_1G P^{\alpha_2}$; the corresponding homogeneous variety $X= G/P$ can be described as the twisted incidence variety of pairs $(l,W)$ where $l \subset k^{2n}$ is a line, $W \subset k^{2n}$ is an isotropic $2$-plane and $F(l) \subset W$. Then the natural morphism $\Lie G \to H^0(X,T_X)$ is not surjective; in particular the automorphism group is non reduced and satisfies
    \[
    G_{\text{ad}} = \PSp_{2n} = (\underline{\Aut}_X^0)_{\text{red}} \subsetneq \underline{\Aut}_X^0 \subsetneq \PGL_{2n}.    \]
\end{example}

The key point for having \emph{enough place} for an infinitesimal fattening of the reduced part of the automorphism group is that we have a bigger ambient group $\hat{G}$, acting non-trivially on the target of the contraction with smooth fibers. For this it is necessary that the target is of the form $G/P^\alpha$ with $(G,\alpha)$ being exceptional; see \Cref{aut main} for a more precise statement. 

\begin{lemma}
\label{lem nonred exceptional}
    Let $(G,\alpha)$ be exceptional in the sense of Demazure and consider a parabolic subgroup of $G$ of the form
    \[
    P = P^\alpha \cap Q, \qquad {}_1G \subset Q,
    \]
    where $Q$ is a non-reduced parabolic subgroup containing $U_{-\alpha}$. Then 
    the group $\underline{\Aut}_X^0$ is \emph{non reduced}.
\end{lemma}

\begin{proof}
By \Cref{lem: reduced aut} above, we know that the reduced part of $\underline{\Aut}^0_X$ is equal to $G$, thus it suffices to show that the dimension of $H^0(X,T_X) = \Lie \underline{\Aut}^0_X$ is strictly bigger than the one of $\Lie G$ to conclude. Let us start by the following diagram,
    \begin{center}
        \begin{tikzcd}
            & X=G/P \arrow[rd, "h"] \arrow[ld, "f_\alpha"]  \arrow[rr, hookrightarrow, "f_\alpha\times h"]& & G/P^\alpha \times G/Q\\
            G/P^\alpha & & G/Q &
        \end{tikzcd}
    \end{center}
    where $f_\alpha$ and $h$ are contractions of Schubert curves and $f_\alpha \times h$ is a closed immersion by \Cref{thm: classification higher}. Consider the short exact sequence of sheaves
    \[
    0 \longrightarrow T_\alpha \longrightarrow T_X \longrightarrow f_\alpha^\ast T_{G/P^\alpha} \longrightarrow 0,
    \]
    where $T_\alpha$ denotes the relative tangent sheaf of $f_\alpha$. We claim that it is split: indeed, let $\mathcal{F}$ be the subsheaf of $T_X$ generated by $\Lie G$, where the action of the Lie algebra is via global vector fields. Since $T_{G/P^\alpha}$ is generated by $\Lie G$, then the image of $\mathcal{F}$ generates $f^\ast_\alpha T_{G/P^\alpha}$.  Moreover, $Q$ contains ${}_1G$ by assumption, so $h$ is ${}_1G$-invariant and its differential is $\Lie G$-invariant. Thus, $\mathcal{F}$ is contained in the relative tangent sheaf $T_h$ of $h$. Since $f_\alpha \times h$ is a closed immersion, $T_\alpha \cap T_h =0$, which means that $\mathcal{F} \to f_\alpha^\ast T_{G/P^\alpha}$ is injective as well. So $\mathcal{F}$ provides the required splitting and we have $T_X = T_\alpha \oplus f_\alpha^\ast T_{G/P^\alpha}$. In particular,
    \begin{align*}
    \dim H^0(X,T_X) & = \dim H^0(X,T_\alpha) \oplus \dim H^0(G/P^\alpha, T_{G/P^\alpha}) \geq \dim H^0(G/P^\alpha,T_{G/P^\alpha})\\
    & = \dim \Lie \hat{G} > \dim \Lie G,
    \end{align*}
    which concludes the proof.
\end{proof}

\begin{corollary}
\label{cor : ghat 1}
    With the assumption of \Cref{lem nonred exceptional}, we have 
    \[
    G = (\underline{\Aut}_X^0)_{\text{red}} \subsetneq {}_1 \hat{G} \cdot G \subset \underline{\Aut}_X^0 \subsetneq \hat{G}.
    \]
    In particular, $H^0(X,T_X) = \Lie \hat{G}$.
\end{corollary}

\begin{proof}
    It suffices here to show that the kernel of the induced morphism $f_\ast \colon \underline{\Aut}_X^0 \rightarrow \hat{G}$ is trivial. If it is not the case, then $H := \ker f_*$ being nontrivial implies that it has a non-trivial Frobenius kernel. Thus ${}_1H$ acts trivially on both $G/P^\alpha$ and on $G/Q$, and since $f_\alpha \times h$ is a closed immersion, it acts trivially on $X$, contradicting the fact that it is a subgroup of automorphisms of $X$.
\end{proof}

The preceding arguments involving the tangent bundle, in the case where the smooth contraction maps onto a variety associated to an exceptional pair, show that the automorphism group is non-reduced and that its infinitesimal part of height one coincides with the full Frobenius kernel of $\hat{G}$. However, they do not explicitly determine the entire automorphism group. For this reason, a different approach is adopted in \Cref{sec : full aut}.



\subsection{The full connected automorphism group}
\label{sec : full aut}

A first step towards the proof of \Cref{aut main intro} result is to show that, in the case where the maximal contraction with smooth fiber has target $G/P^\alpha$ where $(G,\alpha)$ is exceptional in the sense of Demazure, then a certain Frobenius kernel of the bigger group $\hat{G} = \underline{\Aut}_{G/P^\alpha}^0$ acts on $X$.

\begin{lemma}
\label{gemme acts}
Let $X=G/P$, where $P = P^\alpha \cap (\ker \xi) P^\prime$ with $(G,\alpha)$ exceptional in the sense of Demazure. Let $m \geq 0$ be the integer defined by the condition that $\ker \xi $ contains ${}_mG$ but not ${}_{m+1}G$. Then
\[
{}_m \hat{G} \subset \underline{\Aut}_X^0.
\]
\end{lemma}

\begin{proof}
    First, let $\hat{\alpha}$ denote the simple root of $\hat{G}$ satisfying
    \[
    Y = G/P^\alpha = \hat{G}/P^{\hat{\alpha}}.
    \]
    Let us denote as $R_u^-(-)$ the unipotent radical of the opposite parabolic subgroup; then we have that the open cell of $Y$ is isomorphic to both $R_u^-(P^\alpha)$ and to $R_u^-(P^{\hat{\alpha}})$, which is something that is very specific to these exceptional cases. Taking the $m$-th Frobenius kernel, which is an infinitesimal neighborhood of order $p^m$ of the identity element in $\hat{G}$, we get
    \begin{align}
    \label{gemme hat}
    {}_m\hat{G} = {}_m(R_u^-(P^{\hat{\alpha}}) \cdot P^{\hat{\alpha}}) \subset {}_m R_u^-(P^{\hat{\alpha}}) \cdot {}_m P^{\hat{\alpha}} = {}_mR^-_u(P^\alpha) \cdot {}_m P^{\hat{\alpha}}.
    \end{align}
    Next, let us consider $h \in {}_m \hat{G}$. We want to construct the action in such a way that the contractions $f$ and $h$ are $\hat{G}$- (respectively $G$-) equivariant. So let us consider $X$ embedded into the product $\hat{G}/P^{\hat{\alpha}} \times G/(\ker \xi)P^\prime$ via the closed immersion $f \times h$ and let us define the action on the base point $o \in X$ as being
    \[
    h \cdot o = h \cdot (P^{\hat{\alpha}}, {}_mG P^\beta) := (hP^{\hat{\alpha}}, {}_mGP^\beta).
    \]
    To check that this action is well defined, we use the inclusion (\ref{gemme hat}). In particular it tells us that there is some $u \in {}_mR_u^-(P^\alpha) \subset {}_mG$ such that $hP^{\hat{\alpha}} = uP^{\hat{\alpha}}$, and this implies
    \[
    h \cdot o = (uP^{\hat{\alpha}}, {}_mGP^\beta ) = (uP^{\hat{\alpha}}, u\, {}_mGP^\beta) = u \cdot o.
    \]
    This shows that ${}_m\hat{G}$ acts on the variety $X$. In particular, since $f$ is equivariant and ${}_m\hat{G}$ acts faithfully on $Y$, we can conclude that is also acts faithfully on $X$.
\end{proof}

\begin{lemma}
\label{rmk no m2}
Assume that $(G,\alpha)$ is an exceptional pair and consider an infinitesimal subgroup $K \subset \hat{G}$ that is normalized by $G$, contains the Frobenius kernel ${}_m \hat{G}$ but not ${}_{m+1} \hat{G}$. Assume moreover that $K$ acts faithfully on $X=G/(P^\alpha \cap (\ker \xi)P^\prime)$ and trivially on $G/(\ker \xi) P^\prime$, where $\xi$ is either an $m$-th iterated Frobenius or its composition with the very special isogeny $\pi$. Let $K^\prime := K / {}_m \hat{G}$. Then the following hold.
      \begin{enumerate}[(a)]
          \item If $\Lie K^\prime \cap \Lie G^{(m)} = 0$ and $\xi$ is an $m$-th iterated Frobenius, then $K^\prime = 1$.
          \item If the very special isogeny is defined, $\Lie K^\prime \cap \Lie G^{(m)} = \Lie N_{G^{(m)}}$ and $\xi$ is an $m$-th Frobenius followed by $\pi$, then $K^\prime = N_{G^{(m)}}$.
      \end{enumerate}
\end{lemma}

\begin{proof}
Since $\ker \xi$ is assumed to be minimal with respect to inclusion and since the variety $X$ has Picard rank at least two, there is some simple root $\beta \neq \alpha$ such that the smallest subgroup containing both $P^\beta$ and $P^\alpha \cap (\ker \xi)P^\prime$ coincides with $(\ker \xi) P^\beta$. Geometrically, this means that
\[
\psi \colon X = G/(P^\alpha \cap (\ker \xi) P^\prime) \longrightarrow G/(P^\alpha \cap {}_m G P^\beta)
\]
is a contraction, namely the one sending all Schubert curves to a point except for those associated to $\alpha$ and to $\beta$. Consider the complementary contraction with source $X$ which contracts exactly $C_\alpha$ and $C_\beta$ and denote it as $\psi^\prime$, with target $G/Q$. The composition
\[
\psi^\prime \colon X \stackrel{h}{\longrightarrow} G/({}_mG P^\prime) \longrightarrow G/Q,
\]
since $K$ acts trivially on the target of $h$, implies that $K$ also acts trivially on $G/Q$. Since $\psi \times \psi^\prime$ is an $\underline{\Aut}_X^0$-equivariant closed immersion and since $K$ is assumed to act faithfully on $X$, then it must act faithfully on the target of $\psi$. This way we are reducing to treating the case of Picard rank two i.e. \! we can and will assume from now on that
\begin{align}
\label{target psi}
X = G/(P^\alpha \cap (\ker \xi)P^\beta),
\end{align}
where $\alpha$ is an exceptional root and $\ker \xi \in \{{}_m G, \, {}^m N\}$.

\smallskip

\textbf{(a).} Let us assume that we are in the first situation, namely that $\ker \xi = {}_mG$. Then by \Cref{keylemmaG} we have that $\Lie K^\prime \cap \Lie G^{(m)} = 0$ implies directly that $\Lie K^\prime = 0$ (in which case we are done), unless $p=2$ and $G$ is either of type $B_n$ or of type $G_2$. Assume we are in the latter cases, then $\Lie K^\prime \neq 0$ implies that $\Lie K^\prime = \mathfrak{z} (\Lie \hat{G}^{(m)})$ is the one-dimensional center. Now, under these assumptions, the target (\ref{target psi}) of the contraction $\psi$ is exactly one among the families of varieties of Picard rank two considered in \Cref{no mu2 case1} and \Cref{no mu2 case2}. By lifting $\Lie K^\prime$, we get a height one subgroup $\boldsymbol{\mu}_2 \subset K^\prime$, which is isomorphic to $H^{(m)}$ with $H \simeq \boldsymbol{\mu}_{p^{m+1}}$ having Lie algebra $\Lie H = \mathfrak{z}(\Lie \hat{G})$. So we can apply  \Cref{no mu2 case1} and \Cref{no mu2 case2} respectively, which gives a contradiction with the assumption that $K$ acts faithfully on the target of $\psi$.

\smallskip

\textbf{(b).} Let us assume that we are in the second situation, namely that $\ker \xi = {}^m N$. Then by \Cref{keylemmaN} we have that $\Lie K^\prime \cap \Lie G^{(m)} = \Lie N_{G^{(m)}}$ implies that $\Lie K^\prime = \Lie N_{G^{(m)}}$, and then by \Cref{KequalsN} we conclude that $K^\prime = N_{G^{(m)}}$ and we are done. This holds uniformly, unless $p=2$ and $G$ is of type $B_n$. In the latter case, if $\Lie K^\prime \neq \Lie N_{G^{(m)}}$, then 
\[
\Lie K^\prime = \Lie N_{G^{(m)}} \oplus \mathfrak{z}(\Lie \hat{G}^{(m)}).\]
Under these assumptions, the target (\ref{target psi}) of the contraction $\psi$ is exactly one among in the varieties considered in \Cref{no mu2 case3}. Then we proceed as in \textit{(a)} and get again a contradiction with the assumption that $K$ acts faithfully on the target of $\psi$.
\end{proof}

All necessary ingredients are now in place to compute the full automorphism group. The proof proceeds by induction on \( r \), the Picard rank of \( X \), which by \Cref{sec contractions} equals the number of simple roots of \( G \) not contained in the Levi subgroup of \( P_{\text{red}} \).

\begin{lemma}
\label{lem : 2 and Gm}
Assume that $(G,\alpha)$ is exceptional and consider 
\[
P = P^\alpha \cap {}_mG P^\beta
\]
for some $m\geq 1$. Then $\underline{\Aut}_X^0 = {}_m\hat{G}\cdot G$.
\end{lemma}

\begin{proof}
    Let $f_\beta \colon X \to G/{}_mGP^\beta$ denote the contraction associated to the simple root $\beta$, then the induced map by \Cref{blanchard} satisfies
    \begin{center}
    \begin{tikzcd}
    F^m_G \colon G \arrow[r] & \underline{\Aut}_X^0 \arrow[rr, "(f_\beta)_\ast"] && G^{(m)}
    \end{tikzcd}
    \end{center}
    so in particular $(f_\beta)_\ast$ is faithfully flat; denoting $K$ its kernel, we have a short exact sequence
    \begin{center}
        \begin{tikzcd}
            1 \arrow[r] & K \arrow[r] & \underline{\Aut}_X^0 \arrow[rr, "(f_\beta)_\ast"] && G^{(m)} \arrow[r]& 1
        \end{tikzcd}
    \end{center}
    so if we can describe $K$ then we are done. Notice that $K$ acts on the fibers of $f \colon X \to G/P^\alpha$, which means that we can see it via $(f_\alpha)_\ast$ as a subgroup of $\hat{G} = \underline{\Aut}^0_{G/P^\alpha}$. Moreover, since $\underline{\Aut}_X^0$ and $G^{(m)}$ have the same dimension, $K$ is finite. It is also connected: indeed, assume that there is some $H \subset K$ which is a constant subgroup; then in particular it is smooth so it is contained in the reduced part, which is $G$. But then $(f_\beta)_\ast$ cannot be trivial on $H$, because the iterated Frobenius is an isomorphism on constant groups. We also know by \Cref{gemme acts} that $K$ contains the whole $m$-th Frobenius kernel of $\hat{G}$ and moreover that $K \cap G = {}_mG$, because 
    \[
    {(f_\beta)_\ast}_{\vert (\underline{\Aut}_X^0)_{\text{red}}} = {(f_\beta)_\ast}_{\vert G} = F^m_G \colon G \longrightarrow G^{(m)}
    \]
    Thus we can consider the quotient 
    \[
    K^\prime := K/{}_m\hat{G} \subset \underline{\Aut}_X^0/{}_m(\underline{\Aut}_X^0) \subset (\underline{\Aut}_X^0)^{(m)} \subset \hat{G}^{(m)}
    \]
    which is infinitesimal and whose Lie algebra is strictly contained in $\Lie \hat{G}^{(m)}$. Let us denote as $\mathfrak{k}$ its Lie algebra. 
    Since $K^\prime \cap G^{(m)} = 1$, we know that $ \mathfrak{k} \cap \Lie G^{(m)} = 0$. Moreover, the fact that $K$ is a normal subgroup in $\underline{\Aut}_X^0$ implies that $\mathfrak{k}$ is a Lie subalgebra of $\Lie \hat{G}^{(m)}$ and a $G^{(m)}$-submodule. The equality $K \cap G = {}_mG$ implies moreover that $\mathfrak{k} \cap \Lie G^{(m)} = 0$. Then \Cref{rmk no m2} implies that $K^\prime = 1$, 
    which forces $K = {}_m \hat{G}$, as wanted.
\end{proof}

\begin{example}
    Let us illustrate the following proof on a family of examples with $G= \PSp_{2n}$. We denote as
\[
\alpha_1 = \varepsilon_1-\varepsilon_2, \quad \ldots \quad \alpha_{n-1} = \varepsilon_{n-1}-\varepsilon_n, \quad \alpha_n = 2\varepsilon_n
\]
the basis of the root system of $G$, with $\alpha_n$ being the only long simple root. Let us consider a variant of the variety of \cite[Proposition 4.3.4]{BSU}, namely
\[
X_m = \{ (l,W) \colon l \subset k^{2n+1} \text{ line, } W \subset k^{2n} \text{ totally isotropic, } \dim W = 2, \, \text{such that } F^m(l) \subset W\},
\]
where $F^m$ denotes the Frobenius morphism $x_i \mapsto x_i^{p^m}$ on each coordinate of $k^{2n}$. We can write it as $G/P$, where $G = \PSp_{2n}$ and 
\[
P = P^{\alpha_1} \cap {}_m G P^{\alpha_2}.
\]
Here $P^{\alpha_1}$ is the stabilizer of a line in $k^{2n}$, while $P^{\alpha_2}$ is the stabilizer of an isotropic subspace of dimension $2$ in $k^{2n}$. We have the following diagram with the two contractions of Schubert curves
\begin{center}
    \begin{tikzcd}
        & X_m \arrow[rr, hookrightarrow, "f_{\alpha_1} \times f_{\alpha_2}"] \arrow[ddl, "f_{\alpha_1}"] \arrow[ddr, "f_{\alpha_2}"] && G/P^{\alpha_1} \times G/{}_m GP^{\alpha_2} \\
        &&&\\
         G/P^{\alpha_1} = \hat{G}/P^{\alpha} = \proj^{2n-1} && G/{}_m G P^{\alpha_2} &
    \end{tikzcd}
\end{center}
where $\hat{G} = \PGL_{2n}$ and $\alpha$ is the first simple root in type $A_{2n-1}$. Here, the action of the Frobenius kernel ${}_m \hat{G}$ is trivial on $G/{}_mGP^{\alpha_2}$, which is the isotropic Grassmannian of $2$-planes in $k^{2n}$ (with $G$-action pre-composed by an $m$-th iterated Frobenius morphism). On the other hand, ${}_m \hat{G}$ acts faithfully on $\proj^{2n-1}$ and it preserves the condition $F^m(l) \subset W$, thus it also acts on $X_m$.
\end{example}

\begin{lemma}
\label{lem : 2 and Nm}
    Assume that $(G,\alpha)$ is exceptional and that the very special isogeny exists for $G$, and consider 
    \[
    P^\alpha \cap {}^m NP^\beta,
    \]
    for some $m \geq 1$. Then $\underline{\Aut}_X^0 = {}_m\hat{G} \cdot G$.
\end{lemma}

\begin{proof}
    Reasoning in the same way as in the above proof, we get that the map induced by the contraction $f_\beta$ satisfies
    \begin{center}
    \begin{tikzcd}
    F^m_{\overline{G}} \circ \pi_G \colon G \arrow[r] & \underline{\Aut}_X^0 \arrow[rr, "(f_\beta)_\ast"] &&  \overline{G}^{(m)}.
    \end{tikzcd}
    \end{center}
    Denoting $K = \ker (f_\beta)_\ast$, \Cref{gemme acts} implies that both ${}_m\hat{G}$ and ${}^mN$ are contained in $K$. Moreover, by restricting $(f_\beta)_\ast$ to the reduced part of the automorphism group we deduce that $K\cap G = {}^m N$. In order to conclude, it is enough to show that $K$ is the subgroup generated by them. We proceed to consider the quotient
    \[
    K^\prime := K/{}_m\hat{G} \subset \underline{\Aut}_X^0 /{}_m(\underline{\Aut}_X^0) \subset (\underline{\Aut}_X^0)^{(m)} \subset \hat{G}^{(m)},
    \]
    which is again an infinitesimal subgroup by the same reason as in the above lemma. Let us denote as $\mathfrak{k}$ its Lie algebra, which is a $G^{(m)}$-submodule and a Lie subalgebra of $\Lie \hat{G}^{(m)}$, and which satisfies $\mathfrak{k}\cap \Lie G^{(m)} = \Lie N_{G^{(m)}}$ because $K \cap G = {}^m N$. By \Cref{rmk no m2}, we conclude that  
     $K^\prime = N_{G^{(m)}}$, as wanted. 
\end{proof}

The only case left in Picard rank two which is not covered by the two above lemmas (and which is not an exotic $G_2$-case) is the following.

\begin{lemma}
\label{lem : notexceptional}
    Consider a parabolic subgroup of the form
    \[
    P = P_J \cap {}^m N P^\beta \quad\text{or} \quad P = P_J \cap {}_m G P^\beta,
    \]
    where $m\geq 0$ is an integer and where $\beta \in J$. Moreover, assume that $P_J$ is not of the form $P^\alpha$ for any $(G,\alpha)$ exceptional in the sense of Demazure. Then $\underline{\Aut}_X^0 = G$.
\end{lemma}

\begin{proof}
Let us denote as $f_\beta$ the contraction whose target is respectively $G/{}^mN P^\beta$ or $G/{}_mGP^\beta$. Let us start by the first case, namely assuming that $P = P_J \cap {}^mN P^\beta$. By reasoning as in the previous proofs, the map $f_\beta$ induces an homomorphism of algebraic groups $(f_\beta)_\ast$ between the automorphism groups of $X$ and of $G/{}^mNP^\beta$. In particular, we have
\begin{center}
\begin{tikzcd}
    G \arrow[r, hookrightarrow] \arrow[rrr, bend left, "F^m_{\overline{G} \circ \pi_G}"]& \underline{\Aut}_X^0 \arrow[rr, "(f_\beta)_\ast"] && \overline{G}^{(m)} &
    \text{or} & G \arrow[r, hookrightarrow] \arrow[rr, bend left, "F^m_{\overline{G}} \circ \pi_G"] & \underline{\Aut}_X^0 \arrow[r] \arrow[rr, bend right, "(f_\beta)_\ast"] & \overline{G}^{(m)} \arrow[r, "\iota"] & \hat{\overline{G}}^{(m)},
\end{tikzcd}
\end{center}
where $\overline{\iota}$ denotes the inclusion of $\overline{G}$ into $\hat{\overline{G}}$ and this case only happens when $(\overline{G},\overline{\beta}$) is exceptional. Now since $P_J$ is by assumption not exceptional we have that the automorphism group of $G/P_J$ is just $G$; hence 
\[
K = \ker (f_\beta)_\ast = {}^m N \subset G = \underline{\Aut}^0_{G/P_J},\]
so that $(f_\beta)_\ast$ is just the map $F^m_{\overline{G}} \circ \pi_G$ with source $G$, and we are done.

\medskip

Next, let us assume that $P= P_J \cap {}_mGP^\beta$. Then we follow the same reasoning as above: the map $(f_\beta)_\ast$ satisfies 
\begin{center}
\begin{tikzcd}
    G \arrow[r, hookrightarrow] \arrow[rrr, bend left, "F^m_{G}"]& \underline{\Aut}_X^0 \arrow[rr, "(f_\beta)_\ast"] && G^{(m)} &
    \text{or} & G \arrow[r, hookrightarrow] \arrow[rr, bend left, "F^m_{G}"] & \underline{\Aut}_X^0 \arrow[r] \arrow[rr, bend right, "(f_\beta)_\ast"] & G^{(m)} \arrow[r, "\iota"] & \hat{G}^{(m)},
\end{tikzcd}
\end{center}
where $\iota$ denotes the inclusion of $G$ into $\hat{G}$ and this case only happens when $(G,\beta)$ is exceptional. Then we use again the hypothesis that $P_J$ is not exceptional to conclude that
\[
K = \ker(f_\beta)_\ast = {}_m G \subset G = \underline{\Aut}^0_{G/P_J},
\]
so that $(f_\beta)_\ast$ coincides with the map $F^m_G$ with source $G$.
\end{proof}


All the necessary tools to prove \Cref{aut main intro}, whose statement is recalled below, are now in place. 

\begin{theorem}
\label{aut main}
    Consider a parabolic subgroup of a simple adjoint group $G$, of the form
    \[
    P = P_J \cap (\ker \xi)P^\prime,
    \]
     where the roots not in the Levi of $P_J$ and those not in the Levi of $P^\prime_{\text{red}}$ form distinct subsets of the basis. 
    Moreover, assume that $\xi$ is minimal with respect to inclusion. Then the following hold. 
    \begin{enumerate}
        \item If $P_J$ is not of the form $P^\alpha$ for an exceptional root $\alpha$, then $\underline{\Aut}_X^0 = G$.
        \item If $P_J = P^\alpha$ for some exceptional root $\alpha$, let $m \geq 0$ such that $ker \xi$ contains the Frobenius kernel ${}_mG$ but not ${}_{m+1}G$. Then $\underline{\Aut}_X^0 = {}_m\hat{G} \cdot G$, where $\hat{G}$ is the automorphism group of $G/P^\alpha$.
    \end{enumerate}
\end{theorem}

\begin{proof}

Let us proceed by induction on $r \geq 2$, the number of simple roots not in a Levi subgroup of $P_{\text{red}}$, i.e.\! the Picard rank of $X = G/P$. For $r = 2$ the result follows from \Cref{lem : notexceptional}, \Cref{lem : 2 and Nm}, \Cref{lem : 2 and Gm}. Assume that the statement is true for any $r^\prime \geq r$.

\smallskip

\textbf{(1).} Let us place ourselves in the first situation. Namely, the parabolic is either $P = P_J \cap {}_mGP^\prime$ or $P = P_J \cap {}^mNP^\prime$ for some $m \geq 0$, 
and such that $P^\prime_{\text{red}} = P_{J^\prime}$ with $J \cap J^\prime  = I$. Moreover, $P_J$ is not of the form $P^\alpha$ for any exceptional root $\alpha$ and $\Delta \backslash (J \cup J^\prime)$ has cardinality $r+1$. Since $\Delta \backslash J^\prime$ has cardinality $r^\prime \geq r$, we can apply the induction hypothesis to the variety $G/(\ker \xi) P^\prime$, which is either isomorphic to $G/P^\prime$ or isomorphic to  $\overline{G}/\overline{P}^\prime$, where we denote as 
\[
\overline{P}^\prime := (F^m_{\overline{G}}\circ \pi_G)(P^\prime).
\]
Notice that in both cases the automorphism group of the target variety is twisted by $(-)^{(m)}$, because the stabilizer $(\ker \xi)P^\prime$ by assumption contains the Frobenius kernel ${}_m G$. Let us treat the two cases separately.

\smallskip

 Assume that $\ker \xi = {}_mG$. By construction, the parabolic $P^\prime$ is either of the form \textbf{(1)}, or it is of the form 
\[
P^\prime = P^\alpha \cap (\ker \zeta) P^{\prime\prime}.
\]
for some isogeny $\zeta$ and some simple root $\alpha$ such that $(G,\alpha)$ is exceptional. By applying the induction hypothesis to $P^\prime$ we get that
\[
\underline{\Aut}^0_{G/{}_mGP^\prime} = G \quad \text{or} \quad \underline{\Aut}^0_{G/{}_mG P^\prime} = {}_r\hat{G} \cdot G,
\]
where $s$ is such that $\ker \zeta$ contains the $s$-th Frobenius kernel of $G$ but not the $(s+1)$-th. Next, denoting as $f$ and $h$ the two contractions with source $X$ and respective targets $G/P_J$ and $G/{}_mGP^\prime$, we have the following diagram. In the first case, \begin{center}
    \begin{tikzcd}
        1 \arrow[r] & K \arrow[d, "f_\ast", hookrightarrow] \arrow[r, hookrightarrow] & \underline{\Aut}_X^0 \arrow[r, "h_\ast"] & G \arrow[r] & 1\\
        & G = \underline{\Aut}^0_{G/P_J}  \arrow[rru, "F^m_G", bend right] &&&
    \end{tikzcd}
\end{center}
from which we conclude that $K={}_m G$ hence $h_\ast$ is just the $m$-th iterated Frobenius and $\underline{\Aut}_X^0 = G$. In the second case,
\begin{center}
    \begin{tikzcd}
        1 \arrow[r] & K\arrow[r, hookrightarrow] \arrow[d, hookrightarrow, "f_\ast"] & \underline{\Aut}_X^0 \arrow[r, "h_\ast"] & {}_s \hat{G}^{(m)} \cdot G^{(m)} \arrow[r, hookrightarrow] & \hat{G}^{(m)}\\
        & G = \underline{\Aut}^0_{G/P_J} \arrow[rrru, bend right, "\iota"] &&&
    \end{tikzcd}
\end{center}
where $\iota$ is the inclusion of $G^{(m)}$ into $\hat{G}^{(m)}$. In this case, we get again that
\[
K = \ker h_\ast = \ker (\iota \circ F^m_G) ={}_mG.
\]

\smallskip

 Assume that $\ker \xi = {}^mN$ and let us proceed in the analogous way as in the previous case. The parabolic $\overline{P}^\prime$ is either of the form (1) or it is of the form 
\[
\overline{P}^\prime = P^{\overline{\alpha}} \cap (\ker \overline{\zeta}) \overline{P}^{\prime\prime}
\]
for some isogeny $\zeta$ and some simple root $\alpha$ such that $(G,\alpha)$ is exceptional. So by the induction hypothesis 
for some isogeny $\zeta$ and some simple root $\overline{\alpha}$ such that $(\overline{G}, \overline{\alpha})$ is exceptional. Thus, by the induction hypothesis we know that In the first case, by the induction hypothesis 
\[
\underline{\Aut}^0_{G/{}^mNP^\prime} = \overline{G}^{(m)} \quad \text{or} \quad \underline{\Aut}^0_{G/{}^mNP^\prime} = {}_s \hat{\overline{G}}^{(m)}\cdot \overline{G}^{(m)}, \]
where $s$ is such that $\ker \overline{\zeta}$ contains the $s$-th Frobenius kernel of $\overline{G}$ but not the $(s+1)$-th. Next, still denoting by $f$ and $h$ the two contractions with source $X$, we have the following. In the first case,
\begin{center}
    \begin{tikzcd}
        1 \arrow[r] & K \arrow[d, "f_\ast", hookrightarrow] \arrow[r, hookrightarrow] & \underline{\Aut}_X^0 \arrow[r, "h_\ast"] & \overline{G}^{(m)} \arrow[r] & 1\\
        & G = \underline{\Aut}^0_{G/P_J}  \arrow[rru, "F^m_{\overline{G}} \circ \pi_G", bend right] &&&
    \end{tikzcd}
\end{center}
from which we conclude that $K={}^mN$ hence $h_\ast$ is the very special isogeny and $\underline{\Aut}_X^0 = G$. In the second case,
\begin{center}
    \begin{tikzcd}
        1 \arrow[r] & K\arrow[r, hookrightarrow] \arrow[d, hookrightarrow, "f_\ast"] & \underline{\Aut}_X^0 \arrow[r, "h_\ast"] & {}_s \hat{\overline{G}}^{(m)} \cdot \overline{G}^{(m)}  \arrow[r, hookrightarrow] & \hat{\overline{G}}^{(m)}\\
        & G = \underline{\Aut}^0_{G/P_J} \arrow[rrru, bend right, "\overline{\iota} \circ F^m_{\overline{G}} \circ \pi_G"] &&&
    \end{tikzcd}
\end{center}
where $\overline{\iota}$ is the inclusion of $\overline{G}^{(m)}$ into $\hat{\overline{G}}^{(m)}$. In this case, we get that
\[
K = \ker h_\ast = \ker (\iota \circ F^m_{\overline{G}} \circ \pi_G) = {}^mN,
\]
thus once again $h_\ast = F^m_{\overline{G}} \circ \pi_G$, and we are done.

\medskip

\textbf{(2).} Let us place ourselves in the second situation, namely let $P=P^\alpha \cap (\ker \xi) P^\prime$ and assume that $(G,\alpha)$ is exceptional and that $\Delta \backslash J^\prime$ has cardinality $r$, where $P^\prime_{\text{red}} = P_{J^\prime}$. By \Cref{rem: beta not exceptional}, none of the simple roots $\beta \in \Delta \backslash J^\prime$ satisfy that $(G,\beta)$ nor $(\overline{G},\overline{\beta})$ are exceptional in the sense of Demazure. In particular, we can apply the induction hypothesis to $G/(\ker \xi)P^\prime$, which fits into the situation (1). In particular, its automorphism group is either $G$ or $\overline{G}$, acting by a twist of some iterated Frobenius morphism. More precisely, let us denote as $f_\alpha$ and $h$ the contractions with source $X$ and target respectively $G/P^\alpha$ and $G/(\ker \xi)P^\prime$. Then we have the following diagram where the top row is an exact sequence.
\begin{center}
    \begin{tikzcd}
        1 \arrow[r] & K\arrow[r, hookrightarrow] \arrow[d, hookrightarrow, "(f_\alpha)_\ast"] & \underline{\Aut}_X^0 \arrow[r, "h_\ast"] & \underline{\Aut}^0_{G/(\ker \xi)P^\prime} \arrow[r] & 1\\
        & \hat{G} & G = (\underline{\Aut}_X^0)_{\text{red}} \arrow[u, hookrightarrow] \arrow[ur, "\xi", bend right] &&
    \end{tikzcd}
\end{center}
Here either $\xi = F^m_G$ and $\underline{\Aut}^0_{G/(\ker \xi)P^\prime} = G^{(m)}$, or $\xi = F^m_{\overline{G}} \circ \pi_G$ and $\underline{\Aut}^0_{G/(\ker \xi)P^\prime} = \overline{G}^{(m)}$. At this point, we are exactly in the same setting as in the proof of \Cref{lem : 2 and Gm} or of \Cref{lem : 2 and Nm}. By repeating the same reasoning i.e.\! by applying again \Cref{rmk no m2}, we find that $K = {}_m\overline{G}$ in the first case and that $K = {}_m\overline{G}\cdot {}^mN$ in the second case. Thus the automorphism group is indeed isomorphic to ${}_m\overline{G}\cdot G$, which concludes the proof.
\end{proof}

\begin{corollary}
\label{lem : tf nosections}
    Let $G$ be a simple adjoint group and consider a parabolic $P = P_J \cap Q$, with 
    \[
    {}_1G \subset Q
    \] and $Q_{\text{red}}=P_{J^\prime}$, where $J$ and $J^\prime$ are subsets of simple roots such that $J \cap J^\prime = I$. 
    Let $f \colon X=G/P \to G/P_J$. Then the relative tangent bundle $T_f$ satisfies
    \[
    H^0(X,T_f) = 0.
    \]
\end{corollary}

\begin{proof}
    By applying the same reasoning as in the proof of \Cref{lem nonred exceptional}, the condition ${}_1G \subset Q$ ensures that the short exact sequence  
\[
0 \rightarrow T_f \rightarrow T_X \rightarrow f^\ast T_{G/P_J} \rightarrow 0
\]  
is split. Consequently, the tangent bundle of \(X\) satisfies 
\[
T_X = T_f \oplus f^\ast T_{G/P_J}.
\]
Moreover, by \Cref{aut main}, we know that the full automorphism group of \(X\) in this case is either equal to \(G\), or in the case where $P_J = P^\alpha$ is exceptional, then it contains the full Frobenius kernel of $\hat{G}$. Therefore, we obtain  
\[
\dim H^0(X,T_f) = \dim H^0(X,T_X) - \dim H^0(f^\ast T_{G/P_J}) = \dim \Lie \underline{\Aut}^0_X - \dim \Lie G = 0
\] 
in the first case, and
\[
\dim H^0(X,T_f) = \dim H^0(X,T_X) - \dim H^0(f^\ast T_{G/P^\alpha}) = \dim \Lie \underline{\Aut}^0_X - \dim \Lie \hat{G}= 0
\] 
in the second, which completes the proof.
\end{proof}


\section{Type $G_2$ in characteristic two}
\label{sec exotic}
The aim of this section is to prove \Cref{G2 intro}; more precisely, to compute the automorphism group of \( G/P \), where \( P \) is an exotic parabolic subgroup. Throughout, let us assume \( p = 2 \) and let \( G \subset \Sp_6 \) be a simple group of type \( G_2 \).

\subsection{Picard rank one}
By \Cref{th: classification parabolic subgroups}, there exists exactly one homogeneous variety of Picard rank one that is not a standard flag variety. It can be realized as
\[
Y = G/Q \subset X = \Sp_6/P^\beta,
\]
where \( \beta \) denotes the long simple root in type \( C_3 \), and
\[
Q = \mathcal{Q}_2
\]
is the second exotic parabolic subgroup described in \Cref{exotic}. Denote by \( L \) the class of an ample generator of \( \Pic X \), of weight \( \beta \), and by \( M \) the ample generator of \( \Pic Y \), of weight \( 2\alpha_1 + \alpha_2 = \varpi_1 \), where \( \varpi_1 \) is the first fundamental weight of \( G \), and \( \alpha_1, \alpha_2 \) are respectively the short and long simple roots of \( G \). The variety \( Y \) is a general hyperplane section of \( X \) with respect to the line bundle \( L \).

 \begin{proposition}
 \label{rank one exotic}
     The tangent sheaf of the variety $Y$ satisfies $\dim H^0(Y,T_Y) = 14$. In particular, $\underline{\Aut}_Y^0 = G$.
 \end{proposition}

\begin{proof}
    We need the following exact sequences of sheaves
    \begin{align}
        \label{SES1}
        0 \longrightarrow T_X \otimes L^{-1} = T_X(-1) \longrightarrow T_X \longrightarrow T_{X \vert Y} \longrightarrow 0,
    \end{align}
    which comes from the fact that $\mathcal{O}_X(-Y) = L^{-1}$, as well as the normal bundle sequence
    \begin{align}
        \label{SES2}
        0 \longrightarrow T_Y \longrightarrow T_{X \vert Y} \longrightarrow L_{\vert Y} \longrightarrow 0.
    \end{align}

    First, let us check that $ H^0(X,T_X(-1)) =H^1(X,T_X(-1))=0.$ Once this is proven, we have by the long exact sequence associated to (\ref{SES1}) that
   \begin{align}
   \label{claim one}
   H^0(X,T_X) = H^0(Y,T_{X\vert Y}) = \Lie \Sp_6.
   \end{align}
   Here the main point is that $X$ is a flag variety with reduced stabilizer, so we can apply classical results such as Kempf vanishing. In particular, if a $P$-module $M$ admits a filtration by one‑dimensional subquotients of dominant weight, then its higher cohomology groups all vanish. In this setting, dominant is intended with respect to the opposite Borel subgroup $B^-$, as in \cite[II, 4]{Jan}. 
   Notice that $T_X(-1)$ is the homogeneous vector bundle corresponding to the $P^\beta$-module
   \[
    M = (\Lie G/\Lie P^\beta) \otimes k_{-\beta},
   \]
   where we denote by $k_\lambda$ the one-dimensional module with associated weight $\lambda$. Its weights are precisely the roots of $\Sp_6$ which are not in the Levi of $P^\beta$, each shifter by $-\beta$. A direct check shows each such shifted weight is still dominant. Thus we can conclude in particular that $H^1(T_X(-1) = 0$. Moreover, we can identify
   \[
   H^0(X,T_X(-1))= H^0(G \times_{P^\beta} (\Lie G /\Lie P^\beta) \otimes k_{-\beta}) = \mathrm{Hom}_{P^\beta}(k_{\beta}, \Lie G /\Lie P^\beta) = 0,
   \]
   because the weight $\varpi_1 = \varepsilon_1+\varepsilon_2+\varepsilon_3$ does not occur among the weights of $\Lie G / \Lie P^\beta$.

   \medskip


   Next, let us focus on the $G$-module $H^0(Y,L)$. For this, we consider the map
   \[
   g \colon Z = G/P^{\alpha_1} = G/Q_{\text{red}} \longrightarrow Y = G/Q,
   \]
   where $Z$ is a smooth quadric in $\proj^6$ and the map $g$ is finite, purely inseparable of degree four, and which corresponds to the quotient by the unipotent infinitesimal subgroup
   \[
   U_Q^- = R_u(P^{\alpha_1})^- \cap Q \simeq \boldsymbol{\alpha}_2\times \boldsymbol{\alpha}_2.
   \]
  The line bundle $L_{\vert Y}$ is induced by the character $\varpi_1$ on $Q_{\text{red}}= P^{\alpha_1}$; by pullback one identifies $g^\ast L = L_{\varpi_1}$ on $Z$, where we denote by $L_\lambda$ the line bundle corresponding to the weight $\lambda$. Then taking global sections we get
  \[
  H^0(Y,L_{\vert Y}) = H^0(Z,g^\ast L_{\vert Y})^{U^-_Q} = H^0(Z,L_{\varpi_1}) = V(\varpi_1)
  \]
  Next, we can use the fact that the last term is the $7$-dimensional representation of $G$ onto the pure octonions (this follows for example from the Weyl dimension formula). Thus we have proved that 
   \begin{align}
   \label{claim two}
   H^0(Y,L) = V(\varpi_1)
   \end{align}
   as $G$-module. 
   Finally, let us consider the first terms of the long exact sequence associated to (\ref{SES2}). By using (\ref{claim one}) and (\ref{claim two}), together with the fact that $X$ is infinitesimally rigid because it is a standard flag variety, we get 
   \[
   0 \longrightarrow H^0(Y,T_Y) \longrightarrow H^0(X,T_X) = \Lie \Sp_6 \stackrel{\psi}{\longrightarrow} H^0(Y,L) = V(\varpi_1) \stackrel{\sigma}{\longrightarrow} H^1(Y,T_Y) \longrightarrow 0,
   \]
   where all maps are $G$-equivariant. Let us assume that $H^1(Y,T_Y) \neq 0$ which is equivalent to the map $\sigma$ being nonzero. Since the only two nontrivial quotients of $V(\varpi_1)$ as a $G$-module are $V(\varpi_1)$ itself and the simple quotient $L(\varpi_1)$, we get that $\dim H^1(Y,T_Y) \geq 6$. Thus, $\ker \sigma = \mathrm{im} \psi$ has dimension at most one. Since the Lie algebra of $\Sp_6$ infinitesimally moves the hyperplane section $Y$ within its linear system, the map $\psi$ cannot be the zero map. This means that $\dim \ker \psi = \dim \Sp_6 - 1 = 20$. However, $\Lie \Sp_6$ does not contain any $G$-submodule of dimension $20$, so we get a contradiction. This allows us to conclude that $\sigma$ is the zero map, and thus
   \[
   h^0(Y,T_Y) = \dim H^0(X,T_X) - \dim H^0(Y,L) = \dim \Lie \Sp_6 - \dim V(\varpi_1) = 21-7 = 14,
   \]
   which concludes the proof.
    \end{proof}

\subsection{Picard rank two} 
It remains to consider the case of a homogeneous variety of Picard rank two under the group $G$ of type $G_2$. By \cite[Corollary 3.11]{Maccan2}, either $X = G/P$, where $P$ is obtained by fattening via Frobenius kernels and intersections (in which case the automorphism group is already described in \Cref{aut main}), or $X$ is one of the following:
\[
Y_m := G/(\mathcal{Q}_1 \cap {}_m G P^{\alpha_2}), \qquad Z_m := G/(\mathcal{Q}_2 \cap {}_m G P^{\alpha_2}).
\]
 The next statement establishes \Cref{G2 intro}, completing the analysis of the exotic parabolics and providing a full proof of \Cref{thm main geometric}.

\begin{corollary}
\label{G2 rank two}
With the above notation, for $m \geq 1$ the neutral component of the automorphism group of $Y_m$ is
\[
\underline{\Aut}^0_{Y_m} = {}_m(\PGL_6) \cdot G,
\]
while the automorphism group of $Z_m$ is isomorphic to $G$.
\end{corollary}

    \begin{proof}
For the variety $Z_m$, one can proceed as in the proof of \Cref{lem : notexceptional}, replacing the reduced parabolic subgroup $P_J$ with the non-reduced parabolic $\mathcal{Q}_2$. Then the proof follows similarly, using moreover that the automorphism group of $G/\mathcal{Q}_2$ is isomorphic to $G$, as shown in \Cref{rank one exotic}.

\smallskip

For the variety $Y_m$, the strategy is analogous to that in \Cref{lem : 2 and Gm}. Consider the following diagram of contractions of Schubert curves:
\begin{center}
    \begin{tikzcd}
        & Y_m \arrow[dl, "f = f_{\alpha_1}"] \arrow[rd, "h = f_{\alpha_2}"] \arrow[rr, hookrightarrow, "f \times h"]&& \proj^5 \times G/{}_mG P^{\alpha_2}\\
        \proj^5 && G/{}_mGP^{\alpha_2}& 
    \end{tikzcd}
\end{center}
By viewing $Y_m$ as an incidence variety, we see that the $m$-th Frobenius kernel of $\hat{G} := \PGL_6$ acts on it. Indeed, $Y_m$ parametrizes pairs $(l,E)$ where $l \subset W = L(\varpi_1)$ is a line, and $E \subset V = V(\varpi_1)$ is an isotropic subspace of dimension two on which the octonion norm vanishes, with the additional condition that the line $\widetilde{l} \subset V$ lifting $l$ satisfies $\widetilde{l} \subset E$. We define the action of ${}_m \PGL_6$ as the standard action on the first factor and the trivial action on the second; this preserves the incidence relation $F^m(\widetilde{l}) \subset E$. 

\smallskip

Next, consider the kernel $K := \ker h_\ast$ induced by the contraction $h = f_{\alpha_2}$, which by construction contains ${}_m \hat{G}$, and define 
\[
K^\prime := K / {}_m \hat{G}.
\] 
Since $\Lie K^\prime \cap \Lie G^{(m)} = 0,$ we identify 
\[
\Lie K^\prime =: \mathfrak{k} \subset \Lie \PGL_6 / \Lie G.
\]
Recall that our embedding $G \subset \PSp_6 \subset \PGL_6 = \PGL(W)$ arises from the quotient
\[
W = V / k \cdot e,
\]
where \(V\) is the 7-dimensional space of traceless octonions and
\[
e = e_{11} + e_{22}
\]
is the unit element with respect to the octonion product. The Lie algebra \(\operatorname{Lie} \PGL_6\) identifies with the traceless endomorphisms of \(W\). A direct computation shows that as $G$-modules,
\[
\Lie \PSp_6 / \Lie G \simeq V^{(2)} \subset \Lie \PGL_6 / \Lie G \simeq S^2 W.
\]
In particular, the simple subquotients of \(\Lie \PGL_6 / \Lie G\) are the one-dimensional trivial submodule of \(V = V(\varpi_1)\) twisted once by the Frobenius, the simple module \(W\), also twisted by Frobenius, and finally \(\Lambda^2 W\). 

\smallskip

Now, if \(\Lie K^\prime \neq 0\), then \(\Lie K^\prime\) is a nonzero $G$-submodule of \(S^2 W\) that, by the condition \(\Lie K^\prime \cap \Lie G^{(m)} = 0\), must contain a one-dimensional trivial \(G\)-submodule. However, such a trivial submodule does not lift to a submodule inside \(\Lie \PGL_6\), yielding a contradiction. Hence, \(\mathfrak{k} = 0\), so \(K^\prime\) is trivial, and therefore \(K = {}_m \hat{G}\), concluding the proof.

    \end{proof}

 
\bibliographystyle{alpha} 
	\bibliography{biblio} 	
 
\end{document}